\documentclass[12pt,a4paper]{amsart}
\usepackage{amssymb,amsmath}
\usepackage{graphicx}
\usepackage{colordvi}
\usepackage[all]{xy}

\newtheorem{IntroTheorem}{Theorem}
\newtheorem{IntroCor}[IntroTheorem]{Corollary}

\numberwithin{equation}{section}
\newtheorem{theorem}[equation]{Theorem}
\newtheorem{lemma}[equation]{Lemma}
\newtheorem{proposition}[equation]{Proposition}
\newtheorem{cor}[equation]{Corollary}

\theoremstyle{definition}
\newtheorem{definition}[equation]{Definition}

\newtheorem{example}[equation]{Example}

\newtheorem{remark}[equation]{Remark}
\newcounter{FNC}[page]
\def\fauxfootnote#1{{\addtocounter{FNC}{2}$^\fnsymbol{FNC}$%
     \let\thefootnote\relax\footnotetext{$^\fnsymbol{FNC}$#1}}}
\newcommand{\rat}{\ {\relbar\rightarrow}\ }
\newcommand{\longrat}{{\;\relbar\relbar\rightarrow\;}}

\newcommand{\ba}{{\bf a}}
\newcommand{\bb}{{\bf b}}
\newcommand{\bv}{{\bf v}}

\newcommand{\calA}{{\mathcal A}}

\newcommand{\C}{{\mathbb C}}
\renewcommand{\P}{{\mathbb P}}

\newcommand{\Q}{{\mathbb Q}}
\newcommand{\R}{{\mathbb R}}
\newcommand{\Z}{{\mathbb Z}}

\newcommand{\QED}{\hfill\raisebox{-1pt}{\includegraphics{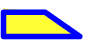}}\vspace{8pt}}
\newcommand{\eqed}{\hfill\raisebox{-1pt}{\includegraphics{figures/QED.eps}}}

\begin{document}

\title{Linear precision for toric surface patches}

\author[H-C. Graf v. Bothmer]{Hans-Christian Graf von Bothmer}
\address{Mathematisches Institut, Georg-August-Universiti\"at G\"ttingen, 
   Bunsenstr. 3-5, D-37073 G\"ottingen Germany}
\email{bothmer@uni-math.gwdg.de}
\urladdr{http://www.crcg.de/wiki/index.php5?title=User:Bothmer}

\author[K.~Ranestad]{Kristian Ranestad}
\address{Matematisk institutt\\
         Universitetet i Oslo\\
         PO Box 1053, Blindern\\
         NO-0316 Oslo\\
         Norway}
\email{ranestad@math.uio.no}
\urladdr{http://www.math.uio.no/\~{}ranestad}
\author[F.~Sottile]{Frank Sottile}
\address{Department of Mathematics\\
         Texas A\&M University\\
         College Station\\
         TX \ 77843-3368\\
         USA}
\email{sottile@math.tamu.edu}
\urladdr{http://www.math.tamu.edu/\~{}sottile}
\thanks{Work of Sottile supported by NSF grants CAREER DMS-0538734 and DMS-0701050,  
        the Institute for Mathematics and its Applications, and Texas Advanced Research
   Program under Grant No. 010366-0054-2007} 

\begin{abstract}
 We classify the homogeneous polynomials in three variables whose toric polar linear
 system defines a Cremona transformation.
 This classification includes, as a proper subset, the classification of toric
 surface patches from 
 geometric modeling which have linear precision.
 Besides the well-known tensor product patches and B\'ezier triangles, we identify a
 family of toric patches with trapezoidal shape, each of which has linear precision.
 Furthermore, B\'ezier triangles and tensor product patches are special cases of
 trapezoidal patches.
 
\end{abstract}
\maketitle

\begin{center}
Communicated by Wolfgang Dahmen and Herbert Edelsbrunner
\end{center}

%
\section*{Introduction} 
%

While the basic units in the geometric modeling of surfaces are B\'ezier triangles and
rectangular tensor product patches, some applications call for multi-sided $C^\infty$
patches (see~\cite{KK00} for a discussion).
Krasauskas's toric B\'ezier patches~\cite{KR02} are a flexible and mathematically
appealing system of such patches.
These are based on real toric varieties from algebraic geometry, may have shape any
polytope $\Delta$ with integer vertices, and they include the classical B\'ezier patches as special
cases. 
For descriptions of multisided patches and toric patches, see~\cite{Gold}.

More precisely, given a set of lattice points in $\Z^n$ with convex hull $\Delta$,
Krasauskas defined toric B\'ezier functions, which are polynomial
blending functions associated to each lattice point. 
This collection of lattice points and toric B\'ezier functions, together with a positive
weight associated to each lattice point is a \Blue{{\sl toric patch}}.
Choosing also a control point in $\R^d$ for each lattice point leads to a map
$\Phi \colon \Delta \to \R^d$ whose image may be used in modeling.
If we choose the lattice points themselves as control points we obtain the 
\Blue{{\sl tautological map}} $\tau \colon \Delta \to \Delta$, which  
is a bijection. 
If the tautological map has a rational inverse, 
then the toric patch has \Blue{{\sl linear precision}}.

The lattice points and weights of a toric patch are encoded in
a homogeneous multivariate polynomial $F(x_0,\dotsc,x_n)$ with positive coefficients,
with every such polynomial corresponding to a toric patch.
In~\cite{GS} it was shown that the toric patch given by $F$ has linear precision
if and only if the associated toric polar linear system,
 \[
    T(F)\ =\  \Bigl\langle\; x_0\frac{\partial F}{\partial x_0},\,
           x_1\frac{\partial F}{\partial x_1},\,
           \dotsc,\,
           x_n\frac{\partial F}{\partial x_n}\;\Bigr\rangle\,,
 \]
defines a birational map $\Blue{\Phi_F}\colon\P^n \rat \P^n$.
This follows from the existence of a rational reparameterization transforming 
the tautological map into $\Phi_F$.
The polar linear system is toric because the derivations $x_i\frac{\partial}{\partial x_i}$
are vector fields on the torus $(\C^\times)^n\subset\P^n$.

When $T(F)$ defines a birational map, we say that $F$ defines a 
\Blue{{\sl toric polar Cremona transformation}}.
We seek to classify all such homogeneous polynomials $F$ without the
restriction that the coefficients are positive or even real.
This is a variant of the classification of homogeneous polynomials $F$ whose polar linear
system 
(which is generated by the partial derivatives $\frac{\partial F}{\partial x_i}$) 
defines a  birational map.
Dolgachev~\cite{Do00} classified all such square free polynomials in 3 variables
and those in 4 variables that are products of linear forms.\medskip

\noindent{\bf Definition.}
 Two polynomials $F$ and $G$ are called \Blue{{\sl equivalent}} if they can be transformed
 into each other by successive invertible monomial substitutions, multiplications with
 Laurent monomials, or scalings of the variables. \medskip

The property of defining a 
toric polar Cremona transformation is preserved under this equivalence.
Our main result is the classification (up to equivalence) of
homogeneous polynomials in three variables that 
define toric polar Cremona transformations.

\begin{IntroTheorem}\label{T:Classification}
 A homogeneous polynomial $F$ in three variables that defines a toric polar
 Cremona transformation is equivalent to one of the following
 \begin{enumerate}
  \item  $(x+z)^a(y+z)^b$ for $a,b\geq 1$,
  \item  $(x+z)^a\bigl( (x+z)^d + yz^{d-1}\bigr)^b$ for $a\geq 0$ and $b,d\geq 1$, or 
  \item  $\bigl(x^2+y^2+z^2-2(xy+xz+yz)\bigr)^d$, for $d\geq 1$.
 \end{enumerate}
\end{IntroTheorem}

When $a=0$ and $d=1$ in (2), we obtain the polynomial
$(x+y+z)^b$, which corresponds to a B\'ezier triangular patch of degree $b$ used in
geometric modeling.
Similarly, the polynomials $F$ in (1) correspond to tensor product patches, which are also  
common in geometric modeling.
These are also recovered from the polynomials in (2) when $d=0$, after multiplying by $z^b$.
Less-known in geometric modeling are \Blue{{\sl trapezoidal patches}}, which 
correspond to the polynomials of (2) for general parameters $a,b,d$.
Their blending functions and weights are given in Example~\ref{Ex:trapezoid}.

\begin{IntroCor}\label{T:Toric_patches}
  The only toric surface patches possessing linear precision are
  tensor product patches, B\'ezier triangles, and the trapezoidal patches of 
  Example~$\ref{Ex:trapezoid}$.
\end{IntroCor}

The polynomials of Theorem~\ref{T:Classification}(3) cannot arise in geometric
modeling, for they are not equivalent to a polynomial with positive coefficients.

We remark that the notion of linear precision used here and
in~\cite{GS} is more restrictive than typically used in geometric
modeling.
There, linear precision often means that there are control points in $\Delta$
so that the resulting map $\Delta\to\Delta$ is the identity. 
We include these control points in our definition of a patch to give a 
precise definition that enables the mathematical study of linear precision.
Nevertheless, this restrictive classification will form the basis for a more
thourough study of the general notion of linear precision for toric
patches.

In Section~\ref{S:LPPP}, we review definitions and results from~\cite{GS}
about linear precision for toric patches, including Proposition~\ref{P:TPLS} which
asserts that a toric patch has linear precision if and only if a
polynomial associated to the patch defines a toric polar Cremona
transformation, showing that Corollary~\ref{T:Toric_patches} 
follows from Theorem~\ref{T:Classification}.
We also show directly that polynomials associated to B\'ezier
triangles, tensor product patches, and trapezoidal patches define
toric polar Cremona transformations.
In particular, this implies that trapezoidal patches have linear precision.
In Section~\ref{S:Basics}, we prove that the above equivalence preserves the
property of defining a toric polar Cremona transformation. 
Then we give our proof of Theorem~\ref{T:Classification}.
Three important ingredients of this proof are established in the
remaining sections.
In Section~\ref{S:contracted}, we show that if all factors of $F$ are contracted, 
then $F$ has two such contracted factors and we identify them.
In Section~\ref{S:not_contracted}, we classify the non contracted factors of $F$,
and we conclude in  Section~\ref{S:technical} with an analysis of possible singularities
of the curve defined by $F$.

Most of our proofs use elementary notions from algebraic
geometry as developed in~\cite{CLO}.
The only exceptions are in Section~\ref{Sec:genus}, where we blow up
a binomial curve to compute its arithmetic genus, and
Section~~\ref{S:technical}, which uses the resolution of base points of
a linear series.\medskip

\noindent {\bf Notation.}  
We shall use the term \Blue{{\sl linear system}} on $\P^2$ both for a vector space 
of forms and for the projective space of curves that they define.  
More generally a linear system on a surface defines a rational map to a
projective space.  
A common factor in the linear system can be removed without changing this map, so we shall
say that two linear systems are equivalent if they define the same rational map.  
For example,
\[
  \langle F,G,H \rangle\ \equiv\ \langle xF, xG,xH \rangle\,.
\]

%
\section{Linear precision and toric surface patches}\label{S:LPPP}
%

We review some definitions and results of~\cite{GS}.
See~\cite{Cox03,KR02,So03} for more on toric varieties and their relation to geometric
modeling.

Let $\Delta\subset\R^n$ be a lattice polytope (the vertices of $\Delta$ lie
in the integer lattice $\Z^n$).
This may be defined by its facet inequalities
\[
  \Delta\ =\ \{s\in\R^n \,:\, h_i(s)\geq 0\,, i=1,\dotsc,N\}\,.
\]
 Here, $\Delta$ has $N$ facets and for each $i=1,\dotsc,N$, 
 $\Blue{h_i}(s):=\langle \bv_i,s\rangle + c_i$ is the linear function defining
 the $i$th facet, where $\bv_i\in\Z^n$ is the (inward oriented) primitive vector normal to 
 the facet and $c_i\in\Z$. 

 Let $\calA\subset\Delta\cap\Z^n$ be any subset of the integer points of
 $\Delta$ which includes its vertices.
 Let $w=\{w_\ba: \ba\in\calA\}\subset \R_>$ be a collection of
 positive weights.
 For each $\ba\in\calA$, Krasauskas defined the \Blue{{\sl toric B\'ezier function}}
 \begin{equation}\label{Eq:toric-Bezier}
   \Blue{\beta_\ba}(s)\ :=\
    w_\ba\cdot h_1(s)^{h_1(\ba)}h_2(s)^{h_2(\ba)}\dotsb h_N(s)^{h_N(\ba)}\,.
 \end{equation}
Then $(\beta_\ba(s) :  \ba\in\calA)$ are the
blending functions for the 
\Blue{{\sl toric B\'ezier patch of shape $(\calA,w)$}}.

Given \Blue{{\sl control points}}
$b=\{\bb_\ba :  \ba\in\calA\}\subset\R^m$ we may define the map 
 \begin{equation}\label{Eq:TBPwCP}
  \begin{array}{rcl}
    \Phi\ \colon\ \Delta &\longrightarrow& \R^m\,,\\
        s&\longmapsto&    
   {\displaystyle \frac{\sum_{\ba\in\calA} \bb_\ba \cdot \beta_\ba(s)}%
        {\sum_{\ba\in\calA} \beta_\ba(s)}}\ .
  \end{array}
 \end{equation}
Precomposing the function $\beta_\ba(s)$ with a homeomorphism $\psi\colon\Delta\to\Delta$
gives a new function $\beta_\ba(\psi(s))$.
Using these new functions in place of the original functions $\beta_\ba$
in~\eqref{Eq:TBPwCP} does not change the shape $\Phi(\Delta)$ but will alter the
parameterization of the patch. 

The toric B\'ezier patch of shape $(\calA,w)$ 
has \Blue{{\sl linear precision}} if the tautological map
 \begin{eqnarray}
    \tau\ \colon\ \Delta &\longrightarrow& \Delta\nonumber\\
        s&\longmapsto&   
   \frac{\sum_{\ba\in\calA} \ba \cdot \beta_\ba(s)}%
        {\sum_{\ba\in\calA} \beta_\ba(s)}\nonumber
 \end{eqnarray}
is the identity.
While this may not occur for the given blending functions, Theorem~1.11 in~\cite{GS}
asserts that there is a unique reparameterization by a homeomorphism
$\psi\colon\Delta\to\Delta$ so that $(\beta_\ba(\psi(s)) :  \ba\in\calA)$ has linear
precision. 
Unfortunately, these new functions $\beta_\ba(\psi(s))$ may not be easy
to compute.
The toric patch of shape $(\calA,w)$ has 
\Blue{{\sl rational linear precision}} if these new functions 
$\beta_\ba(\psi(s))$ are rational functions or polynomials.
This property has an appealing mathematical reformulation.

Given data $(\calA,w)$ as above, suppose that 
$d:=\max\{\Blue{|\ba|}:=a_1+\dotsb+a_n :  \ba\in\calA\}$ is
the maximum degree of a monomial $x^\ba$ for $\ba\in\calA$.
Define the homogeneous polynomial
\[
    \Blue{F_{\calA,w}}(x_0,x_1,\dotsc,x_n)\ :=\ 
    \sum_{\ba\in\calA} w_\ba\; x_0^{d-|\ba|} x^\ba\,.
\]
The \Blue{{\sl toric polar linear system}} of $F_{\calA,w}$
is the linear system generated by its toric derivatives,
 \begin{equation}\label{Eq:TPLS}
    \Blue{T(F_{\calA,w})}\ :=\ 
      \Bigl\langle\; x_0\frac{\partial F_{\calA,w}}{\partial x_0},\,
           x_1\frac{\partial F_{\calA,w}}{\partial x_1},\,
           \dotsc,\,
           x_n\frac{\partial F_{\calA,w}}{\partial x_n}\;\Bigr\rangle\,.
 \end{equation}
%

\begin{proposition}[Corollary 3.13 of~\cite{GS}]\label{P:TPLS}
  The toric patch of shape $(\calA,w)$ has rational linear precision if and only if
  its toric polar linear system~$\eqref{Eq:TPLS}$ defines a birational isomorphism
  $\P^n \rat \P^n$.
\end{proposition}

We illustrate Proposition~\ref{P:TPLS} through some examples of patches with
linear precision. 

\begin{example}[B\'ezier curves]\label{Ex:curves}
 Let $\calA:=\{0,1,\dotsc,d\}\subset\R$.
 If we choose weights $w_i:=\binom{d}{i}$, then the toric B\'ezier functions are
 \begin{equation}\label{Eq:Bez_Curve}
   \beta_i(s)\ :=\ \tbinom{d}{i} s^i (d-s)^{d-i}\,,
   \qquad i=0,1,\dotsc,d\,.
 \end{equation}
 The polynomial is
\[
   F_{\calA,w}\ =\ \sum_{i=0}^d \tbinom{d}{i} x^i y^{d-i}\ =\ (x+y)^d\,,
\]
 and its associated toric polar linear system is
\[
  T(F_{\calA,w})\ =\ \left\langle xd(x+y)^{d-1},\, yd(x+y)^{d-1}\right\rangle
  \ \equiv\ \langle x, y\rangle\,,
\]
which defines the identity map $\P^1\to\P^1$.
Thus the toric patch with shape $(\calA,w)$ has rational linear precision.

Substituting $s=d\cdot t$ and removing the common factor of $d^d$, the 
toric B\'ezier functions~\eqref{Eq:Bez_Curve} become the univariate Bernstein
polynomials, the blending functions for B\'ezier curves.
Up to a coordinate change, this is the only toric patch in dimension 1
which has rational linear precision~\cite[Example 3.15]{GS}.
More precisely, the toric polar linear system of a homogeneous polynomial $F(x,y)$ that is
prime to $xy$ defines a birational isomorphism $\P^1\to \P^1$ if and only if $F$ is the
pure power of a linear form that does not vanish at either coordinate point $[0:1]$ and
$[1:0]$.  
\eqed
\end{example}
We specialize to the case $n=2$ for the remainder of this paper.
Our homogeneous coordinates for $\P^2$ are $[x:y:z]$ and the toric polar linear system
defines the map
 \begin{equation}\label{Eq:TPCP}
   [\;x\ :\ y\ :\ z\;]\ \longmapsto\ 
   \Bigl[\; x\frac{\partial F_{\calA,w}}{\partial x}\ :\ 
            y\frac{\partial F_{\calA,w}}{\partial y}\ :\ 
            z\frac{\partial F_{\calA,w}}{\partial z}\;\Bigr]\,.
 \end{equation}

\begin{example}[Quadratic Cremona Transformation]\label{Ex:cremona}
 Before giving examples of toric surfaces patches with linear precision,
 we describe the classical quadratic Cremona transformation, a birational map
 on the projective plane.
 This is defined by
 \[
    \varphi\ \colon\   [\;x\ :\ y\ :\ z\;]\ \longmapsto\ 
   [\;yz\ :\ zx\ :\ xy\;]\,.
 \]
 At points where $xyz\neq0$, we have
 $\varphi([x:y:z])=[\frac{1}{x}:\frac{1}{y}:\frac{1}{z}]$,
 which shows that $\varphi$ is an involution.
 The map $\varphi$ is undefined at the three coordinate points 
 $[1:0:0]$, $[0:1:0]$, and $[0:0:1]$, which are its \Blue{{\sl basepoints}}.
 For $xy\neq 0$ we have $\varphi([x:y:0])=[0:0:xy]=[0:0:1]$, and so the map $\varphi$
 \Blue{{\sl contracts}} the coordinate line $z=0$ to the point $[0:0:1]$.
 The other coordinate lines are also contracted by $\varphi$.
 Furthermore, as $\varphi([1:ty:tz])=[t^2yz:tz:ty]=[tyz:z:y]$, we see that if $y,z$ are fixed but
 $t$ approaches zero, then $\varphi([1:ty:tz])$ approaches $[0:z:y]$.
 Thus the map $\varphi$ blows up the basepoint $[1:0:0]$ into the line $x=0$.

 We call this map, or any map obtained from it by linear substitution, a
 standard quadratic Cremona transformation.
 There is a second, non standard, quadratic Cremona given by
 $[x:y:z]\mapsto[x^2 :  yz : xz]$.
 We leave the computation of its contracted curves and the resolution of its basepoints as
 an exercise for the reader.
\end{example}

\begin{example}[Tensor product patches]\label{Ex:tensor}
Let $a,b$ be positive integers.
Let $\calA$ be the integer points in the $a\times b$ rectangle,
$\{(i,j) :  0\leq i\leq a,\ 0\leq j\leq b\}$, 
and select weights $w_{i,j}:=\binom{i}{a}\binom{j}{b}$.
Then the corresponding toric B\'ezier functions are
 \begin{equation}\label{Eq:TPP}
   \beta_{i,j}(s,t)\ :=\ \tbinom{i}{a}\tbinom{j}{b}
         s^i(a-s)^{a-i} t^j(b-t)^{b-j}\,,
 \end{equation}
and the homogeneous polynomial is
\[
   F_{\calA,w} = \sum_{i=0}^a\sum_{j=0}^b \tbinom{i}{a}\tbinom{j}{b}x^iy^jz^{a+b-i-j}
    \ =\ (x+z)^a(y+z)^b\,.
\]
Removing the common factor $(x+z)^{a-1}(y+z)^{b-1}$ from the partial
derivatives of $F_{\calA,w}$ shows that
 \begin{multline*}
  \qquad T(F_{\calA,w})\ \equiv\ \langle  ax(y+z),\ by(x+z),\ z(a(y+z) \,+\,b(x+z))\rangle\\
  =\  \langle  (x+z)(y+z),\ z(x+z),\  z(y+z)\rangle\,, \qquad
 \end{multline*}
%
%
which defines a quadratic Cremona transformation with base points
\[
   \bigl\{ [1:1:-1],\ [1:0:0],\  [0:1:0]\bigr\}\,.
\]
By Proposition~\ref{P:TPLS}, this patch~\eqref{Eq:TPP} has rational linear precision.
This is well-known, as after a change of coordinates, these blending functions define
the tensor product patch of bidegree $(a,b)$, which has rational linear precision.
\eqed
\end{example}

\begin{example}[B\'ezier triangles]
Let $\calA:=\{(i,j) :  0\leq i,j\mbox{ and } i+j\leq d\}$.
These are the integer points inside the triangle below.
 \[
   \begin{picture}(112,66)(-31,-7)
    \put(-1,-1){\includegraphics{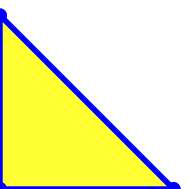}}
    \thicklines
    \put(-31,48){$(0,d)$}  
    \put( 54,-3){$(d,0)$}
    \put(-31,-3){$(0,0)$}   
   \end{picture}
 \]
If we select weights $\Blue{w_{i,j}}:=\frac{d!}{i!j!(d-i-j)!}$, the toric B\'ezier functions are 
 \begin{equation}\label{Eq:BT}
   \beta_{i,j}(s,t)\ :=\ \tfrac{d!}{i!j!(d-i-j)!}
              s^it^j(d-s-t)^{d-i-j}\,,
 \end{equation}
and the homogeneous polynomial is
\[
   F_{\calA,w} = \sum_{i+j+k=d} \tfrac{d!}{i!j!k!}x^iy^jz^k
    \ =\ (x+y+z)^d\,.
\]
Its toric polar linear system is
\[
   T(F_{\calA,w}) = \left\langle xd(x+y+z)^{d-1}, yd(x+y+z)^{d-1}, zd(x+y+z)^{d-1}\right\rangle
   \ \equiv\ \langle x, y, z\rangle\,,
\]
which defines the identity map $\P^2\to\P^2$.
Thus the patch with blending functions~\eqref{Eq:BT} has rational linear precision.
These blending functions are essentially the standard bivariate Bernstein polynomials,
which are used in B\'ezier triangles and have linear precision.
\eqed
\end{example}

\begin{example}[Trapezoids]\label{Ex:trapezoid}
Let $b,d\geq 1$ and $a\geq0$ be integers, and set 
\[
  \calA\ :=\ \{ (i,j)\,:\, 0\leq j\leq b\ 
    \mbox{ and }\  0\leq i\leq a+db-dj\}\,,
\]
which are the integer points inside the trapezoid below.
\[
  \begin{picture}(238,53)(-34,-8)
    \put(-3,-3){\includegraphics{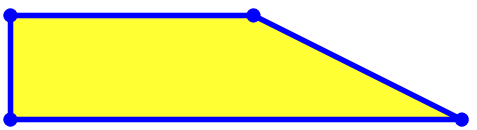}}
    \put( 80,30){$(a,b)$}
    \put(-30,-3){$(0,0)$}   
    \put(135,-3){$(a+db,0)$}
    \put(-30,30){$(0,b)$}  
  \end{picture}
\]
Choose weights  $\Blue{w_{i,j}}:=\binom{b}{j}\binom{a+db-dj}{i}$.
Then the toric B\'ezier functions are
 \begin{equation}\label{Eq:Trapezoid}
   \beta_{i,j}(s,t)\ :=\ \binom{b}{j}\binom{a+db-dj}{i}
       s^i(a+db-s-dt)^{a+db-dj-i} t^j(b-t)^{b-j}\,,
 \end{equation}
and the homogeneous polynomial  is
\[
   F_{\calA,w}\ =\ \sum_{j=0}^b\sum_{i=0}^{a+db-dj} \tbinom{b}{j}\tbinom{a+db-dj}{i}
      x^i y^j z^{a+db-i-j}\ =\ 
    (x+z)^a\bigl( (x+z)^d + y z^{d-1}\bigr)^b\,.
\]
The partial derivatives of $F_{\calA,w}$ have common factor
$(x+z)^{a-1}\bigl( (x+z)^d + y z^{d-1}\bigr)^{b-1}$.
%
%
%
%
%
%
%
%
Removing this and performing some linear algebra shows that 
 \[
    T(F_{\calA,w})\ \equiv\ \left\langle  (x+z)^{d+1},\ yz^{d-1}(x+z),\ 
    z( (a+bd)(x+z)^d + ayz^{d-1} )\right\rangle\,.
 \]
This has a base point at $[1:0:-1]$ of multiplicity 1 and 
one at $[0:1:0]$ of multiplicity $d$.
To see that it defines a birational map, work in the affine chart where $x+z\neq 0$,
and assume that $x=1-z$.
Then the corresponding rational map is
\[
  (y,z)\ \longmapsto\ (yz^{d-1}, (a+bd)z^d+yz^{d-1})\,.
\]
Changing coordinates, this is $(y,z)\mapsto(yz^{d-1},z)$,
which is a bijection when $z\neq 0$.
\eqed
\end{example}

\begin{remark}\label{R:trapezoid}
 The first three patches are widely used and implemented in CAD software.  
 The trapezoid patch reduces to the B\'ezier triangle when $a=0$ and $d=1$, and to the
 tensor product patch when $d=0$.
 While the trapezoid patch for general parameters has not been used explicitly in modeling, 
 special cases of it have appeared implicitly.
 For example, a rational ruled surface in $\R^3$ of degree $2a+d$ with directrix of
 minimal degree $a$ and general one of degree $a+d$~\cite[\S5.2]{PW01} is the image of such
 a patch (here, $b=1$). 
 B\'ezier quad patches on a sphere bounded by circular arcs of minimal type
 $(2,4)$~\cite{Kr06} are also trapezoidal.
 Some quad patches on rational canal surfaces~\cite{Kr07} can be represented by trapezoidal
 patches with $b=2$. 
 The full possibilities for the use of the trapezid patch in modeling have yet to be
 developed. 
\QED
\end{remark}

%
\section{Toric polar Cremona transformations}\label{S:Basics}
%

We classify toric surface  ($n=2$) patches with rational linear precision
through the algebraic relaxation of classifying the homogeneous polynomials 
(\Blue{{\sl forms}}) $F=F(x,y,z)\in \C[x,y,z]$ 
whose toric polar linear system defines a birational map $\P^2\rat\P^2$.
Write $F_x$ for $\frac{\partial}{\partial x}F$, and the same for the other variables
$y$ and $z$.
We will write $F=0$ or simply $F$ for the reduced  curve defined by $F$ in $\P^2$.

\begin{definition}
 Let $F$ be a form.
 The vector space $\Blue{T(F)}:=\langle xF_x, yF_y,zF_z \rangle$ defines the 
 \Blue{{\sl toric polar linear system}} of curves on $\P^2$ and the 
 \Blue{{\sl toric polar map}}
 \begin{eqnarray*}
    \Blue{\varphi_F}\ \colon\ \P^2 &\longrat& \P^2\\
        {} [\;x\ :\ y\ :\ z\;]&\longmapsto& [\;xF_x\ :\ yF_y\ :\ zF_z\;]\,,
 \end{eqnarray*}
 which maps curves in $T(F)$ into lines in the target $\P^2$.
We say that $F$ defines a 
\Blue{{\sl toric polar Cremona transformation}} if
 this map is birational. 
\end{definition}
We establish some elementary properties of such forms $F$ and linear 
systems $T(F)$, and then give our classification of forms that define toric
polar Cremona transformations.

%
%
\subsection{Equivalence of forms}
There are some transformations which send one form defining a toric polar Cremona
transformation into another such form. 
Those which are invertible define an equivalence relation on forms, and our classification
is up to this equivalence. 

\begin{lemma}\label{L:power}
 A form $F$ defines a toric polar Cremona transformation if and only if every power $F^a$
 for $a>0$ defines a toric polar Cremona transformation.
\end{lemma}

\noindent{\it Proof.}
 The toric polar linear systems of $F$ and $F^a$ are equivalent,
\[
   \left\langle x \tfrac{\partial F^a}{\partial x},\;
                y \tfrac{\partial F^a}{\partial y},\;
                z \tfrac{\partial F^a}{\partial z}\right\rangle\ =\ 
   \langle axF^{a-1}F_x,\, ayF^{a-1}F_y,\,azF^{a-1}F_z\rangle \ \equiv\ 
   \langle xF_x,\, yF_y,\,zF_z\rangle\,.
    \eqno{\QED}
\]

The linearity of differentiation implies that $F(x,y,z)$ defines a toric polar Cremona  
transformation if and only if $F(ax,by,cz)$ defines a toric polar Cremona  
transformation, for all non zero $a,b,c\in\C$.
Call this \Blue{{\sl scaling the variables}}.
Multiplication by a monomial also preserves the property of defining a toric polar Cremona
transformation. 

\begin{lemma}\label{L:monomial}
  A form $F$ defines a toric polar Cremona transformation if and only if
  $x^ay^bz^cF$ defines a toric polar Cremona transformation, for any positive integers
  $a,b,c$. 
\end{lemma}

\begin{proof}
  It suffices to check that $T(xF)\equiv T(F)$.
  Note that $T(xF)$ is
\[
      \left\langle x \tfrac{\partial}{\partial x} xF,\;
                y \tfrac{\partial}{\partial y} xF,\;
                z \tfrac{\partial}{\partial z} xF\right\rangle\ =\ 
   \langle xF + x^2F_x,\, yxF_y,\,zxF_z\rangle \ \equiv\ 
   \langle xF_x,\, yF_y,\,zF_z\rangle\,,
\]
  which is $T(F)$.
  The last equivalence follows by removing the common factor $x$ and
  applying the Euler relation, which is $xF_x+yF_y+zF_z=\deg(F)F$.
\end{proof}

The calculations in the proof of Lemma~\ref{L:monomial}
hold when the exponents $a,b,c$
are \Blue{{\sl any}} integers.
Consequently, $F$ may be any homogeneous Laurent polynomial.
For example,
\[
   y^{-1}z^{-1} + x^{-2}yz^{-1} + x^{-2}y^{-1}z 
   - 2 ( x^{-1}z^{-1} + x^{-1}y^{-1} + x^{-2} )
\]
is a Laurent form defining a toric polar Cremona transformation.
(This is the form of Theorem~\ref{T:Classification}(3) with $d=1$ multiplied by the
monomial $x^{-2}y^{-1}z^{-1}$.) 

A third class of transformations are the invertible monomial transformations.
A vector $\alpha=(\alpha_1,\alpha_2,\alpha_3)\in\Z^3$ corresponds to a (Laurent) monomial
$\Blue{t^\alpha}:=x^{\alpha_1}y^{\alpha_2}z^{\alpha_3}$ of degree $\Blue{|\alpha|}:=\alpha_1+\alpha_2+\alpha_3$.
Let $\alpha,\beta,\gamma\in\Z^3$ be three exponent vectors and consider the map
$\C^3\to\C^3$ defined by
 \begin{equation}\label{Eq:monom_map}
    (x,y,z)\ \longmapsto\ (t^\alpha,t^\beta,t^\gamma)\,.
 \end{equation}
%

\begin{lemma}\label{L:monomial trans}
 The formula~$\eqref{Eq:monom_map}$ defines a rational map $\P^2\rat\P^2$ if and only if\/
 $|\alpha|=|\beta|=|\gamma|$. 
 This map is invertible if and only if $\alpha-\gamma$ and $\beta-\gamma$ form a basis for 
 $\{v\in\Z^3  :  |v|=0\}$.
\end{lemma}

We prove Lemma~\ref{L:monomial trans} later.
Suppose that $\calA:=\{\alpha,\beta,\gamma\}\subset\Z^3$ satisfies the hypotheses of
Lemma~\ref{L:monomial trans} so that the map $\Blue{\varphi_\calA}\colon\P^2\rat\P^2$
defined by~\eqref{Eq:monom_map} is a birational isomorphism.
It induces a map $\Blue{\varphi^*_\calA}$ on monomials $x^ay^bz^c$ by
 \begin{equation}\label{Eq:imt}
   \varphi^*_\calA(x^ay^bz^c)\ :=\ t^{\alpha a+\beta b+\gamma c}\ =:\ 
    \Blue{t^{\Blue{\calA {\bf a}}}}\,,
 \end{equation}
 where ${\bf a}:=(a,b,c)^T$ and $\calA {\bf a}$ is the multiplication of the 
 vector ${\bf a}$ by the matrix $\calA$ whose columns are $\alpha,\beta,\gamma$.
 When $\calA\in\mbox{Mat}_{3\times 3}\Q$ is invertible and satisfies the hypothesis of
 Lemma~\ref{L:monomial trans}, we call $\varphi^*_\calA$ an 
 \Blue{{\sl invertible monomial transformation}}. 
 Under the hypotheses of Lemma~\ref{L:monomial trans}, the condition that $\calA$ is
 invertible is equivalent to $|\alpha|\neq 0$.

\begin{lemma}\label{L:imt}
  A form $F$ defines a toric polar Cremona transformation if and only if
  $\varphi^*_\calA(F)$ does for any invertible monomial transformation $\varphi^*_\calA$. 
\end{lemma}

\begin{proof}
 By~\eqref{Eq:imt}, the toric derivative 
 $x\frac{\partial}{\partial x}\varphi^*_\calA(t^{\bf a})$
 is 
\[
   (A_1\cdot {\bf a}) t^{\calA {\bf a}}\ =\ 
  \varphi^*_\calA((A_1\cdot {\bf a})t^{\bf a})
  \ =\  A_1\cdot \varphi^*_\calA( x\tfrac{\partial}{\partial x}t^{{\bf a}},\,
                          y\tfrac{\partial}{\partial y}t^{{\bf a}},\,
                          z\tfrac{\partial}{\partial z}t^{{\bf a}})^T\,,
\]
 where 
 $A_1$ is the first row of the matrix $\calA$.
 Thus
 \begin{eqnarray*}
     \left(  x\tfrac{\partial}{\partial x}\varphi^*_\calA(t^{{\bf a}}),\,
             y\tfrac{\partial}{\partial y}\varphi^*_\calA(t^{{\bf a}}),\,
             z\tfrac{\partial}{\partial z}\varphi^*_\calA(t^{{\bf a}}) \right)^T
     &=&
     \calA\left(\varphi^*_\calA( x\tfrac{\partial}{\partial x}t^{{\bf a}},\,
                          y\tfrac{\partial}{\partial y}t^{{\bf a}},\,
                          z\tfrac{\partial}{\partial z}t^{{\bf a}})^T\right)\,,\\
  &=& \varphi^*_\calA( x\tfrac{\partial}{\partial x}t^{{\bf a}},\,
                          y\tfrac{\partial}{\partial y}t^{{\bf a}},\,
                          z\tfrac{\partial}{\partial z}t^{{\bf a}})^T\,.
 \end{eqnarray*}
 as $\calA$ is invertible.
 Thus we have the relation between the toric polar linear systems
\[
   T( \varphi^*_\calA(F))\ =\ \varphi^*_\calA(T(F))
\]
 for any homogeneous polynomial $F$.
 The lemma follows as $\varphi_\calA$ is birational.
\end{proof}

\begin{definition}
 Let $\Blue{\C[x^{\pm}]}:=\C[x,x^{-1}, y,y^{-1}, z,z^{-1}]$ be the ring of Laurent
 polynomials, the coordinate ring of the torus $(\C^*)^3$.
 This is $\Z$-graded by the total degree of a monomial.
 Forms $F,G\in\C[x^{\pm}]$ are \Blue{{\sl equivalent}} if 
 $G=\psi^*(F)$, where $\psi$ is a composition of
\begin{itemize}
  \item[(a)] scaling variables, $[x:y:z]\mapsto [ax:by:cz]$, or
  \item[(b)] multiplication by a monomial, or
  \item[(c)] an invertible monomial transformation.
\end{itemize}
 Our classification is up to this equivalence.
 By (b), it is no loss to assume that a Laurent form $F$ is an ordinary form
 (in $\C[x,y,z]$).
\end{definition}

\begin{proof}[Proof of Lemma~$\ref{L:monomial trans}$]
 The first statement is clear as a rational map $\C^3\to\C^3$ drops to a map
 $\P^2\rat\P^2$ if and only if it is defined by homogeneous rational forms of the same
 degree. 
 For the second, consider the map $(\C^*)^3\to\P^2$ defined by
\[
  \varphi\ \colon\ t:=(x,y,z)\ \longmapsto\ [t^\alpha\,:\,t^\beta\,:\,t^\gamma]\,,
\]
 and suppose that we have $s,t\in(\C^*)^3$ with $\varphi(s)=\varphi(t)$. After rescaling
 in the source, we may assume that $t=(1,1,1)$.
 In particular, $s$ is a solution of
\[
   s^\alpha\ =\ \lambda\,,\qquad
   s^\beta\ =\ \lambda\,,\qquad\mbox{and}\qquad
   s^\gamma\ =\ \lambda\,,
\]
 for some $\lambda\in\C^*$.
 But we also have
 \begin{equation}\label{Eq:u_eq}
   s^{\alpha-\gamma}\ =\ 1\ =\ s^{\beta-\gamma}\,.
 \end{equation}
 Since $|\alpha|=|\beta|=|\gamma|$, solutions to~\eqref{Eq:u_eq} include the diagonal
 torus $\Delta:=\{(a,a,a) :  a\in\C^*\}$, so we see again that $\varphi$ is defined on
 the dense torus $(\C^*)^3/\Delta$ of $\P^2$.
 This map on the dense torus is an isomorphism if and only if points of the diagonal torus
 are the only solutions to~\eqref{Eq:u_eq},
 which is equivalent to the condition that the exponents $\alpha-\gamma$ and
 $\beta-\gamma$ are a basis for the free abelian group
 $\{v\in\Z^3 :  |v|=0\}$.
\end{proof}

%
%
\subsection{Proof of Theorem~\ref{T:Classification}}
Let $F$ be a form defining a toric polar Cremona transformation.
We first classify the possible irreducible factors of $F$, and then 
determine which factors may occur together.
The classification of the factors of $F$ occupies
Sections~\ref{S:contracted}, \ref{S:not_contracted}, and~\ref{S:technical}.

Under the rational map $\varphi_F$ each component of each curve in the toric polar linear
system is either contracted (mapped to a point) or mapped to a dense subset of a curve.
As $F\in T(F)$, this in particular holds for the curve $F=0$, whose
components correspond to the irreducible factors of $F$.
(We always take the reduced structure on this curve, that is,  we consider the set of zeroes
of $F$, not the scheme defined by $F$.)
An irreducible factor of $F$ is \Blue{{\sl contracted}}, respectively 
\Blue{{\sl not contracted}}, if the corresponding component of the curve $F=0$ is
contracted by the linear system $T(F)$, respectively not contracted. 
There are three possibilities for the factors of $F$.
 \begin{equation}\label{Eq:Trilemma}
  \begin{minipage}[c]{12cm}
     \begin{enumerate}
      \item $F$ has no contracted factors, or 
      \item $F$ has only contracted factors, or 
      \item $F$ has both contracted and non contracted factors.
     \end{enumerate}
  \end{minipage}
 \end{equation}

We get information about the factors of $F$ from a  simple, but useful restriction lemma.

\begin{lemma}\label{L:restriction}
 Suppose that $G$ is an irreducible factor of a form $F$.
 Then the restrictions of the toric polar linear systems of $F$ and of $G$ to the curve
 $G=0$ are equivalent.
\end{lemma}

\noindent{\it Proof.}
 Write $F = G^nH$ with $G$ and $H$ coprime. 
 Then 
 \[ 
   T(F)\ =\ 
    \langle nxG^{n-1}G_xH + xG^nH_x, \ 
            nyG^{n-1}G_yH + yG^nH_y, \ 
            nzG^{n-1}G_zH + zG^nH_z \rangle\,.
 \]
 After factoring out $G^{n-1}$, restricting to $G=0$, and factoring out $nH$ we obtain
 \[
    T(F)|_G\ \equiv \ \langle xG_x, yG_y, zG_z \rangle|_{G}\ =\  T(G)|_G\,.\eqno{\QED}
 \]
%

A cornerstone of our classification is that there is a strong restriction on the 
singularities of the curve $F=0$.
A singular point $p$ of a curve is \Blue{{\it ordinary}} if locally near $p$, the curve
consists of $r>1$ smooth branches that meet transversally at $p$.
In Section~\ref{S:technical} we prove the following theorem.
\medskip

\begin{theorem}\label{Th:singularities}
  If a form $F$ coprime to $xyz$ defines a toric polar
  Cremona transformation, then the curve $F=0$ has at most one singular point
  outside the coordinate lines, and if there is such a point, then all factors of $F$ are
  contracted.  Furthermore, if this singular point is ordinary,
 then $F$ has at most two distinct factors through the singular point.
\end{theorem}

Since the toric polar map $\varphi_F$ is birational,
each component of a curve in the toric polar linear system $T(F)$ is either contracted 
or mapped birationally onto a line, and at most one component of a 
curve is not contracted.
This in particular holds for $F$.

\begin{cor}\label{Co:non_contr}
  At most one factor of $F$ is not contracted.
\end{cor}

In Section~\ref{S:not_contracted} we classify the possible non contracted factors of $F$.

\begin{theorem}\label{T:not_contracted}
 If $F$ is an irreducible form defining a curve with no
 singularities outside the coordinate lines whose toric polar linear
 system maps this curve birationally onto a line,
 then $F$ is equivalent to one of the following forms,
 \begin{enumerate}

  \item $x^2+y^2+z^2-2(xy+xz+yz)$, or

  \item $(x+z)^d + yz^{d-1}$, for some integer $d\geq 1$.

 \end{enumerate}
\end{theorem}

Example~\ref{Ex:trapezoid} with $a=0$ and $b=1$ shows that the 
second class of forms define toric polar Cremona
transformations, and the following example shows that the first class also does.


\begin{example}\label{Ex:relaxation}
 The form $F$ of Theorem~\ref{T:not_contracted}(1) has the toric polar linear system 
 \begin{eqnarray}
   T(F)& =& \langle x^2-xy-xz,\ y^2-xy-yz,\ z^2-xz-yz\rangle \nonumber\\
   &=&\langle \Blue{(x{-}y{-}z)}\Brown{(y{-}x{-}z)},\ 
 \label{eq:strange_equality}
             \Blue{(x{-}y{-}z)}\Purple{(z{-}x{-}y)},\ 
             \Purple{(y{-}x{-}z)}\Brown{(z{-}x{-}y)}\rangle\ , 
 \end{eqnarray}
 which defines a quadratic Cremona transformation with base points
\[
   \bigl\{ [1:0:1],\ [0:1:1],\  [1:1:0]\bigr\}\,.
\]
 To help see the equality~\eqref{eq:strange_equality}, note that
 $-xF_x-yF_y+zF_z= \Blue{(x-y-z)}\Brown{(y-x-z)}$.

 This example shows that the algebraic relaxation (seeking polynomials $F$ with arbitrary
 coefficients whose toric
 polar linear system is birational) of the original problem from geometric modeling 
 has solutions which do not come from geometric modeling, as the coefficients of $F$
 cannot simultaneously be made positive. \eqed 
\end{example}

In Section~\ref{S:contracted}, we study the possible contracted factors of $F$
and prove the following lemma.

\begin{lemma}\label{L:contr}
  Suppose that $G$ is a contracted factor of $F$.
  Then, up to a permutation of variables, $G=x^a+\alpha y^bz^{a-b}$ with $a$, $b$
  coprime and $\alpha\neq 0$.
\end{lemma}

We now present a proof of Theorem~\ref{T:Classification},
following the three cases of~\eqref{Eq:Trilemma}.

%
\subsubsection{$F$ has no contracted factors}
%
In this case, Corollary~\ref{Co:non_contr} implies that $F$ has a single irreducible
factor.
Since this factor is not contracted  it is equivalent 
to one of the forms described in Theorem~\ref{T:not_contracted}.
Since these forms define toric polar Cremona transformations, Lemma~\ref{L:power} implies
that any power of such a form defines a toric polar Cremona transformation.
In particular, $F$ defines a toric polar Cremona transformation.
This establishes part (2) of Theorem~\ref{T:Classification}(2) when
$a=0$ and also part (3).

%
\subsubsection{$F$ has only contracted factors}
%
We outline the proof in this case, which is carried out in
Section~\ref{S:contracted}. 
Suppose first that $F$ defines a toric polar Cremona transformation.
By Lemma~\ref{L:contr}, every contracted factor is a binomial.
We first show that any two contracted factors of $F$ are simultaneously equivalent to 
\[
   (z+x)\qquad\mbox{and}\qquad(z+y)\,,
\]
and in particular they meet outside the coordinate lines.

Suppose that all factors of $F$ are contracted, then we show that $F$ has at least two 
irreducible factors.
We next show that if $F$ has three or more factors, then we may assume that they intersect
transversally at $[1:1:-1]$.   
Hence $F=0$ has an ordinary singularity of multiplicity at least $3$, which contradicts
the last part of Theorem \ref{Th:singularities}. 
Therefore  $F$ is equivalent to
 \begin{equation}\label{F:two_factors}
   (x+z)^a(y+z)^b\,,
 \end{equation}
with $a,b>0$.
By Example~\ref{Ex:tensor}, any such form defines a
toric polar Cremona transformation,
which completes the proof of Theorem~\ref{T:Classification}(1).

%
\subsubsection{$F$ has both contracted and non contracted factors}
%
By Corollary~\ref{Co:non_contr}, $F$ has a unique
non contracted factor.
It also has a unique contracted factor.
Indeed, any two contracted factors meet outside the coordinate
lines, and so the curve $F=0$  is singular outside the coordinate lines.
Then Theorem~\ref{Th:singularities} implies that all factors of $F$ are contracted, 
a contradiction. 
All that remains is to examine the different possibilities for
the factors of $F$.
We show that the non contracted factor cannot be equivalent to $x^2+y^2+z^2-2(xy+xz+yz)$.
We then show that if the non contracted factor is equivalent to $(x+z)^d + yz^{d-1}$, then
(after putting it into this form) the contracted factor is $x+z$.
Example~\ref{Ex:trapezoid} shows that all possibilities
\[
   (x+z)^a \bigl((x+z)^d + yz^{d-1} \bigr)^b
\]
with $a\geq 0$ and $b>0$ define toric polar Cremona transformations, 
which completes the proof of Theorem~\ref{T:Classification}.
These claims about the non contracted factors are proven in the following two lemmas.

\begin{lemma}
  If $F$ has a non contracted factor equivalent to 
  $x^2+y^2+z^2-2(xy+xz+yz)$, then it has no other factors.
\end{lemma}

\begin{proof}
 Suppose that $F$ has two factors, $Q:=x^2+y^2+z^2-2(xy+xz+yz)$, and a contracted factor
 $G$.
 Since permuting the variables does not change $Q$, Lemma~\ref{L:contr} implies that 
 $G=x^A+\alpha y^a z^{A-a}$, with $A,a\geq 0$ coprime and $\alpha$ non zero.

 By Theorem~\ref{Th:singularities}, $G$ and $Q$ can meet only on the coordinate lines.
 If we substitute the parameterization $[x:y:z]=[(s+t)^2:s^2:t^2]$ of $Q$ into
 $G$, we obtain
 \begin{equation}\label{Eq:subst}
    (s+t)^{2A}\ +\ \alpha s^{2a} t^{2A-2a}\,.
 \end{equation}
%
 The condition that $Q$ and $G$ meet only on the coordinate lines implies that the only
 factors of~\eqref{Eq:subst} are $s$, $t$, or $s+t$.
 But this implies that $A=a=0$, contradicting our assumption that $G$
 was a non trivial
 factor of $F$.
\end{proof}

\begin{lemma}
  If $F$ has the  non contracted factor $(x+z)^d + yz^{d-1}$, then its other irreducible
  factor must be $x+z$. 
\end{lemma}

\begin{proof}
 Suppose that $F$ has two factors, $Q:=(x+z)^d + yz^{d-1}$ with $d\geq 1$ and a contracted
 factor $G$, which is necessarily a binomial.
 Multiplying $G$ by a monomial and scaling, we may assume it has the form
 $\alpha+(-1)^ax^Ay^az^{-A-a}$ with $A\geq 0$ and $A,a$ coprime.
 By Theorem~\ref{Th:singularities}, $G$ and $Q$ can meet only on the coordinate lines.
 We solve $Q=0$ for $y$ to obtain $y=-(x+z)^d/z^{d-1}$ and then substitute this into 
 $G$ to obtain
\[
   \alpha+(-1)^ax^A\left(-\tfrac{(x+z)^d}{z^{d-1}}\right)^a z^{-A-a}
   \ =\ \alpha + x^A(x+z)^{ad}z^{-A-ad}\,.
\]
 If we multiply this by $z^{A+ad}$ if $a\geq 0$ and by $z^A(x+z)^{-ad}$ if $a<0$ (and replace
 $a$ by $-a$), this becomes either
\[
   \alpha z^{A+ad}\ +\ x^A (x+z)^{ad}\qquad\mbox{or}\qquad
   \alpha z^A(x+z)^{ad}\ +\ x^A z^{ad}\,.
\]
 Since $G$ and $Q$ can meet only on the coordinate lines, the only possible factors of
 these polynomials are $x$, $z$, and $x+z$.
 Neither $x$ nor $z$ can be a factor as $A\neq ad$ and the coefficients are non zero,
 so $x+z$ is the only factor.
 But then we must have $a=0$, $\alpha=1$, and $A=1$,
 so that $G=1+xz^{-1}$, or, clearing the denominator, $G=x+z$.
\end{proof}

%
\section{contracted factors}\label{S:contracted}
%

We study the contracted factors of a form $F$ that defines a toric polar Cremona
transformation.
We first prove Lemma~\ref{L:contr}, that any contracted factor of $F$ is a binomial.

\begin{proof}[Proof of Lemma~$\ref{L:contr}$]
  By Lemma~\ref{L:restriction}, the restrictions of the toric polar maps of $F$ and of $G$
  to the curve $G=0$ coincide.
  Let $T(G)$ be the toric polar linear system of $G$.  
  Since it contracts $G$, $T(G)|_G=\langle xG_x, yG_y, zG_z \rangle|_G$ is
  one-dimensional, and so $T(G)$ is only two-dimensional.
  Thus there is a linear relationship among the toric derivatives of $G$,
\[
    q_1 xG_x\,+\, q_2 yG_y \,+\, q_3 zG_z \ =\ 0\,.
\]
  Writing $G=\sum_i m_i$ as a sum of terms $m_i= \alpha_i x^{a_i}y^{b_i}z^{c_i}$,
  we see that $q_1a_i+q_2b_i=q_3c_i$ for all $i$.
  Thus we may assume that $q_1,q_2,q_3\in\Z$ and they are coprime.
  Permuting variables, we may assume that the $q_j$ are non negative.
  Since $G$ is homogeneous, say of degree $d$, we have $a_i+b_i+c_i=d$, and so
\[
   (q_1+q_3)a_i + (q_2+q_3)b_i\ =\ q_3 d\,.
 \]
  Thus $G(x,y,1)$ is a weighted homogeneous polynomial of degree $q_{3}d$. 
  Since $G$ is irreducible, the only possibilities are $G(x,y,1)=x$ or
  $G(x,y,1)=y$ (neither can occur as $F$ is coprime to $xyz$), or
  $G(x,y,1)=x^a + \alpha y^b$ with $a$ and $b$ coprime, $\alpha\neq 0$ and  
  $(q_{1}+q_{3})a = (q_{2}+q_{3})b$.
  Since $G$ is irreducible, $z$ does not divide $G$ and $G(x,y,1)$ must have degree $d$. 
  Therefore after possibly interchanging $x$ and $y$ we see that
  $G$ has the form claimed.
\end{proof}

Any two contracted factors may be put into a standard form.

\begin{lemma}\label{L:two_factors}
  If $G$ and $H$ are two contracted factors of $F$, then, up to equivalence
  $GH=(x+z)(y+z)$.
\end{lemma}

In particular, any two contracted factors of $F$ meet outside the coordinate lines.

\begin{proof}
  Up to a permutation of the variables, each factor is a prime binomial of the form 
\[
    x^A + \alpha y^az^{A-a}\,,
\]
 where $A,a\geq 0$ are coprime and $\alpha$ is non zero.
 
 By Theorem~\ref{Th:singularities}, $F$ has at most one singularity outside the
 coordinate lines.
 Points common to two factors of $F$ are singular, so 
 the factors $G$ and $H$ define curves that meet at most once outside
 the coordinate axes.
 To study such points, we dehomogenize and set $z=1$.
 Multiplying $G$ and $H$ by monomials, we may suppose that they have the form
 \begin{equation}\label{Eq:GandH}
  G\ =\ \alpha\ +\ x^Ay^a\qquad\mbox{and}\qquad
  H\ =\ \beta\ + \ x^By^b\,.
 \end{equation}
 Since these are irreducible binomials, $1=\gcd\{|A|,|a|\}=\gcd\{|B|,|b|\}$,
 and since they are coprime $Ab-Ba\neq 0$.

 The points common to the two components are the solutions to $G=H=0$.
 The number of solutions to such a zero-dimensional binomial system is
 $|Ab-Ba|$~\cite[\S~3.2]{Stu}.
 Since there can be at most one such point,
 $|Ab-Ba|=1$.
 Interchanging the roles of $(A,a)$ and $(B,b)$ if necessary, we may assume that 
 $Ab-Ba=1$.
 Under the invertible substitution $x=x^by^{-a}$ and $y=x^{-B}y^A$,~\eqref{Eq:GandH}
 becomes $\alpha+x$ and $\beta+y$.
 Scaling $x$ and $y$ and rehomogenizing, 
 we may assume that the binomials are $x+z$ and $y+z$.
\end{proof}

\begin{lemma}\label{L:irreducible}
  If $F$ has only a single irreducible factor, then that factor is not contracted.
\end{lemma}

In particular a form with all factors contracted must have at least two factors. 

\begin{proof}
  Let $G$ be the irreducible factor of $F$, then $F=G^a$ for some $a>1$.
  By Lemma~\ref{L:power} the toric polar map $\varphi_F$ of $F$ coincides with the toric
  polar map $\varphi_G$ of $G$. 
  If $G=0$ is contracted, then, as in the proof of  Lemma~\ref{L:contr}, 
  $T(G)|_G$ is one-dimensional and thus $T(G)$ is only two-dimensional
  so that $xG_x$, $yG_y$ and $zG_z$ are dependent. 
  But then $\varphi_G$, and hence $\varphi_F$, cannot be birational.
\end{proof}

We classify forms $F$ defining a toric polar Cremona
transformation with all factors contracted.

\begin{theorem}\label{T:all_contr}
  If all factors of $F$ are contracted, then $F$ is equivalent to
  $(x+z)^a(y+z)^b$, for some $a,b>0$.
\end{theorem}

\begin{proof}
 Suppose that all factors of $F$ are contracted.
 By Lemma~\ref{L:two_factors}, we may assume that two of the irreducible factors of $F$
 are $x+z$ and $y+z$.
 We only need to show that there are no further contracted factors of $F$.
 Suppose there is another contracted factor.
 After multiplying by a monomial, this will have the form
\[
   \gamma\ +\ x^Cy^cz^{-C-c}\,,
\]
 with $\gamma\in\C^*$ and $C>1$.
 By Theorem~\ref{Th:singularities}, $F$ has at most one singularity outside of the
 coordinate lines, and so this factor must meet the other factors only in the 
 point $[1:1:-1]$ where they meet.
 It follows that $\gamma=\pm 1$ and $C=|c|=1$.
 
 We see that the only possible irreducible factors of $F$ are
\[ 
  x+z\,,\quad y+z\,,\quad z^2-xy\,,\quad\mbox{ and }\quad
  y-x\,.
\]
Since these four factors have distinct tangents at $[1:1:-1]$, the singularity of $F$ at
this point is ordinary. 
By the last part of Theorem \ref{Th:singularities}, $F$ can have at most two distinct
factors through $[1:1:-1]$ 
so the theorem follows.
\end{proof}

%
\section{Irreducible polynomials}\label{S:not_contracted}
%

We classify irreducible factors of $F$ which are not contracted by the toric polar Cremona
transformation.
Specifically, we prove the following theorem.
\medskip

\noindent{\bf Theorem~\ref{T:not_contracted}.}
{\it 
 If $F$ is an irreducible form defining a curve with no
 singularities outside the coordinate lines whose toric polar linear
 system maps this curve birationally onto a line,
 then $F$ is equivalent to one of the following forms,
 \begin{enumerate}

  \item $x^2+y^2+z^2-2(xy+xz+yz)$, or

  \item $(x+z)^d + yz^{d-1}$, for some integer $d\geq 1$.

 \end{enumerate}
}\medskip

Since the curve $F=0$ is rational, it has a parameterization
$\Blue{\gamma}\colon \P^1\to\P^2$ which determines $F$ up to a constant.
The composition of $\gamma$ with the toric polar Cremona transformation of $F$ is a map 
$\P^1\to \P^2$ of degree 1.
We will deduce from this and the location of the singularities of $F$ that there are exactly
three distinct irreducible factors appearing in $\gamma$. 

Applying quadratic Cremona transformations 
puts $\gamma$ into a standard form from which the 
hypothesis on the singularities of $F=0$ restricts $F$ to be equivalent to one of the forms
of Theorem~\ref{T:not_contracted}.
An important technical part of this argument is the local contribution to the arithmetic
genus of a singular point of a binomial curve, which we compute in Section~\ref{Sec:genus}.

%
\subsection{Linear factors in $\gamma$}
%
Suppose that $F$ is a form of degree $d$ that satisfies the hypotheses of
Theorem~\ref{T:not_contracted} and let $\gamma:=[f:g:h]\,\colon\P^1\to\P^2$ parameterize
the curve $F=0$.
Then $f$, $g$, and $h$ are are coprime forms of degree $d$ on $\P^1$.
Because the toric polar map $\varphi_F$ sends the image of $\gamma$ (the curve $F=0$)
isomorphically onto a line, the map $\P^1\to\P^2$ with components
 \begin{equation}\label{E:components}
   f F_x(\gamma),\  gF_y(\gamma),\ \mbox{ and }\ \  hF_z(\gamma)
 \end{equation}
has degree 1, and thus these forms become linear after removing common factors.
We study their syzygy module to show that there are only three distinct irreducible
factors in $fgh$.\smallskip 

Choose homogeneous coordinates $[s,t]$ on $\P^1$ with $st$ coprime to
$fgh$. 
As $\gamma$ parameterizes $F$, we have $F(\gamma)\equiv 0$ on $\P^1$.
Differentiating with respect to $s$ and $t$ gives
 \begin{equation}\label{Eq:first_relation}
   \left[\begin{matrix}f_s&g_s&h_s\\f_t&g_t&h_t\end{matrix}\right]
   \left[\begin{matrix}F_x(\gamma)\\F_y(\gamma)\\F_z(\gamma)\end{matrix}\right]
   \ =\ \left[\begin{matrix}0\\0\end{matrix}\right]\ .
 \end{equation}
Using the Euler relations $sf_s+tf_t=df$ (and the same for $g$ and $h$) gives the
syzygy
 \[
    f F_x(\gamma)\ +\ gF_y(\gamma)\ +\ hF_z(\gamma)\ =\ 0\,.
 \]
Multiplying the first row of~\eqref{Eq:first_relation} by $fgh$ gives a second syzygy, so
we have 
 \begin{equation}\label{Eq:syzygies}
   \left[\begin{matrix}f_sgh&fg_sh&fgh_s\\1&1&1\end{matrix}\right]
   \left[\begin{matrix}fF_x(\gamma)\\gF_y(\gamma)\\hF_z(\gamma)\end{matrix}\right]
   \ =\ \left[\begin{matrix}0\\0\end{matrix}\right]\ .
 \end{equation}
%
An equivalent set of syzygies is given by the rows of the matrix
 \begin{equation}\label{Eq:other_syzygies}
   \left[\begin{matrix}f_sgh-fgh_s&fg_sh-fgh_s&0\\1&1&1\end{matrix}\right]\ .
 \end{equation}
Since the three components~\eqref{E:components} share a common factor whose removal yields
linear forms $(p,q,r)$ with the same syzygy matrix, $r=-p-q$ and 
the removal of common factors from the first row of~\eqref{Eq:other_syzygies}
gives the syzygy $(-q,p,0)$.

There are three sources for common factors of the first row of~\eqref{Eq:other_syzygies}.
 \begin{enumerate}
  \item Common factors of $f$ and $f_s$, of $g$ and $g_s$, or of $h$ and $h_s$,

  \item common factors of some pair of $f$, $g$, or $h$, and

  \item common factors of $f_sh-fh_s$ and $g_sh-gh_s$.
 \end{enumerate}

   A common factor of the third type that is not of type (1) or (2) vanishes at a point $p\in\P^1$ where 
 \[
  \mbox{\rm rank} \left[\begin{matrix}f_s&g_s&h_s\\f&g&h\end{matrix}\right]
   \ \leq\ 1\,.
 \]
 The Euler relation implies that we also have 
 \begin{equation}\label{E:t_differential}
  \mbox{\rm rank} \left[\begin{matrix}f_s&g_s&h_s\\tf_t&tg_t&th_t\end{matrix}\right]
   \ \leq\ 1\,,
 \end{equation}
 and so $t$ is a common factor of the third type.
 Suppose now that $t(p)\neq 0$.
 Then~\eqref{E:t_differential} shows that 
 the differential of $\gamma$ does not have full rank at $p$,
 and so the the curve $F=0$ is singular at $\gamma(p)$.
 But such a singular point must lie on a coordinate line of $\P^2$ and so 
 one of $f$, $g$, or $h$ vanishes at $p$.
 Without loss of generality, suppose that $f(p)=0$.
 Then $f_s(p)h(p)=0$, as $f_sh-fh_s$ vanishes at $p$,
 and so the common factor vanishing at $p$ divides either $f_s$ or $h$, and
 is therefore a factor of type (1) or (2).
 Thus $t$ is the only factor of type (3) that is not of type (1) or (2).
 As $fgh$ is coprime to $t$, the common
 factor $t$ has multiplicity 1.

Now suppose that $\ell$ is a linear factor of $fgh$ with $\ell^a$, $\ell^b$, and $\ell^c$
exactly dividing $f$, $g$, and $h$, respectively.
Then $\ell^{a+b+c}$ exactly divides $fgh$ and $\ell^{a+b+c-1}$ exactly divides the entries
in the first row of~\eqref{Eq:other_syzygies}.
It follows that if the prime factorization of $fgh$ is $\ell_1^{a_1}\dotsb \ell_k^{a_k}$,
then the common factor of the first row of~\eqref{Eq:other_syzygies} is
\[
    t \ell_1^{a_1-1} \dotsb \ell_k^{a_k-1}\,.
\]
This has degree $3d+1-k$.
Since the entries in the first row of~\eqref{Eq:other_syzygies} have degree $3d-1$, and 
removing this common factor gives linear forms, we have that
\[
   3d-1\ =\ 1+3d+1-k\,,
\]
or $k=3$.
Thus there are exactly three distinct linear factors dividing $fgh$.

%
\subsection{Arithmetic genus of binomial germs}\label{Sec:genus}
%

We compute $\delta_{(0,0)}$, the contribution at the origin to the arithmetic genus of a
curve $C$ with germ 
 \begin{equation}\label{E:binomial_germ}
   (x^a\: -\: y^b)\; u(x,y)\ =\ 0\,,
 \end{equation}
where $u(0,0)\neq 0$.
Write $p_a(C)$ for the arithmetic genus of a curve $C$.

\begin{lemma}\label{L:binomial_germ}
  Let $C$ be a curve on a smooth surface $S$ with germ~$\eqref{E:binomial_germ}$
  given in local coordinates $(x,y)$ of a point $p\in S$.
  If $\widetilde{C}$ is obtained from $C$ by a sequence of blowups in $p$ and points
  infinitely near $p$ and is smooth at all points infinitely near $p$, then
 \begin{equation}\label{E:arith_genus}
   p_a(\widetilde{C})\ =\ p_a(C)\ -\ 
   \tfrac{1}{2}\bigl( (a-1)(b-1) + \gcd(a,b) - 1  \bigr)\,.
 \end{equation}
\end{lemma}

We call the difference $p_a(C)-p_a(\widetilde{C})$ the 
\Blue{{\sl $\delta$-invariant of $C$ at $p$}}.

\begin{proof}
Recall the formula~\cite[Cor V.3.7]{Hart} for the arithmetic genus of
the strict transform $C'$ of a curve $C$ obtained by blowing up a point of multiplicity
$a$ in $C$,
 \begin{equation}\label{E:genus_contribution}
      p_a(C')\ =\ p_a(C)\ -\ \tfrac{1}{2}a(a-1)\,.
 \end{equation}

Let $C$ be defined near $p$ by~\eqref{E:binomial_germ}.
Then $C$ has multiplicity  $\min\{a,b\}$ at $p$.
If $\min\{a,b\}=1$, then $C=\widetilde{C}$ and~\eqref{E:arith_genus}
becomes $p_a(\widetilde{C})=p_a(C)$.

If $a=b$, then $C$ consists of $a$ smooth branches meeting pairwise transversally at $p$.
Blowing up $p$ separates these branches so that $C'=\widetilde{C}$.
By~\eqref{E:genus_contribution}, we have
\[
   p_a(C')\ =\ p_a(C)\ -\ \tfrac{1}{2}a(a-1)\ =\ 
    p_a(C)\ -\ \tfrac{1}{2}\bigl((a-1)(a-1) + a -1\bigr)\,,
\]
which establishes the lemma in this case.

We complete the proof by induction on the maximum of the
exponents of $x$ and $y$.
Suppose that $a<b$.
Then $C$ is tangent to the curve $y=0$ at $p$ and so to compute the blowup $C'$, we 
substitute $x=xy$ in~\eqref{E:binomial_germ} to obtain
\[
   y^a \; (x^a-y^{b-a})\,\cdot\,u(xy,y)\,.
\]
The exceptional divisor ($y=0$) has multiplicity $a$, and the curve $C'$ has local
equation $(x^a-y^{b-a}) u'$, where $u'(x,y)=u(xy,y)$ and so $u'(0,0)\neq 0$.
Since $a,b-a<b=\max\{a,b\}$, our induction hypothesis applies to $C'$ to give
\[
   p_a(\widetilde{C})\ =\ p_a(C')\ -\ 
   \tfrac{1}{2}\bigl( (a-1)(b-a-1)+\gcd(a,b-a) -1  \bigr)\,.
\]
Using~\eqref{E:genus_contribution}, this becomes
 \begin{eqnarray*}
   p_a(\widetilde{C})&=& p_a(C)\ -\ \tfrac{1}{2}a(a-1)\ -\ 
   \tfrac{1}{2}\bigl( (a-1)(b-a-1)+\gcd(a,b-a) - 1  \bigr)\\
    &=&  p_a(C)\ -\ \tfrac{1}{2}\bigl( (a-1)(b-1) + \gcd(a,b)-1  \bigr)\,,
 \end{eqnarray*}
 which completes the proof.
\end{proof}

%
\subsection{Classification}
%

Suppose that $\gamma\colon\P^1\to\P^2$ parameterizes the curve $F=0$.
We may assume that the three linear forms dividing components of $\gamma$ are $s$, $t$,
and $\Blue{\ell}:=-(s+t)$. 
Since the components are relatively prime forms of degree $d$, there are seven
possibilities for $\gamma$, up to permuting components and factors.\medskip 

\begin{tabular}{rlcrlcrl}
 I   & $[s^at^{d-a} : t^b\ell^{d-b} : s^{d-c}\ell^c]$ & \ &
 II  & $[s^at^b\ell^{d-a-b} :  s^{d-c}t^c :  \ell^d]$  & \ &
 III & $[s^at^{d-a} : t^b\ell^{d-b} : \ell^d]$   \\
 IV  & $[s^at^{d-a} : s^{d-b}t^b : \ell^d]$       & \ &
 V   & $[s^at^b\ell^{d-a-b} :  t^d :  \ell^d]$ & \ &
 VI  & $[s^at^{d-a} : t^d : \ell^d]$        \\
 VII & $[s^d : t^d : \ell^d]$\smallskip
\end{tabular}

We assume that all exponents appearing here are positive.

\begin{theorem}\label{T:gamma_classification}
 Suppose that $\gamma$ is a curve with parameterization one of the types {\rm I}---{\rm VII}.
 \begin{enumerate}
  \item If $\gamma$ has type {\rm I} and is smooth outside the coordinate lines, then
         $\gamma$ is equivalent to either
\[
    [s^2t^2\,:\, t^2\ell^2\,:\,  s^2\ell^2]
    \qquad\mbox{or}\qquad
    [s^{d-1}t\,:\, t^{d-1}\ell\,:\, s^{d-1}\ell]\,.
\]
  \item If $\gamma$ does not have type {\rm I}, then it may be transformed into a curve of
        type {\rm I} via quadratic Cremona transformations. 

\end{enumerate}
\end{theorem}

We deduce Theorem~\ref{T:not_contracted} from
Theorem~\ref{T:gamma_classification}.

\begin{proof}[Proof of Theorem~$\ref{T:not_contracted}$]
Suppose that $\gamma$ has the first form in Theorem~\ref{T:gamma_classification}(1).
Apply the standard Cremona transformation  $[x:y:z]\mapsto [yz:xz:xy]$ to $\gamma$  and
remove the common factor $s^2t^2\ell^2$ to obtain
\[
   [s^2t^2\ell^4 \,:\, s^4t^2\ell^2 \,:\, s^2t^4\ell^2]\ =\ 
   [\ell^2\,:\, s^2\,:\, t^2]\,,
\]
which satisfies $x^2+y^2+z^2-2(xy+xz+yz)=0$, the curve in 
Theorem~\ref{T:not_contracted}(1).

Suppose that $\gamma$ has the second form in Theorem~\ref{T:gamma_classification}(1).
Set $a:=d-1$ to obtain $[s^at:t^a\ell:s^a\ell]$.
If $a=0$, this parameterizes the line $x+y+z=0$.
If $a>0$, apply the standard Cremona transformation and 
multiply the $y$-coordinate by $(-1)^a$ to obtain
\[
  [s^at^a\ell^2 : (-1)^as^{2a}t\ell : s^at^{a+1}\ell]\ =\ 
  [t^{a-1}\ell : -(-s)^a : t^a]\,.
\]
Since $\ell=-(s+t)$, we have
\[
   (x+z)^a\ =\ (-st^{a-1})^a\ =\ (-s)^a (t^a)^{a-1}\ =\ -yz^{a-1}\,.
\]
This gives all curves of the form in Theorem~\ref{T:not_contracted}(2).
\end{proof}

We prove Theorem~\ref{T:gamma_classification} in the
following subsections. 

%
%
\subsection{Curves of type I}

Suppose that $\gamma=[s^at^{d-a}:t^b\ell^{d-b}: s^{d-c}\ell^c]$ is a rational curve of type I.
If $\gamma$ parameterizes a curve satisfying the hypotheses of
Theorem~\ref{T:not_contracted}, then it can be singular only on the coordinate lines.
Since all six exponents appearing in $\gamma$ are positive, these singularities 
occur at the coordinate points.
As $\gamma$ is rational, its arithmetic genus
must equal the sum of its $\delta$-invariants at these singular points. 

In the neighborhood of the coordinate point $[0:0:1]$, the curve has germ 
$(x^b-y^{d-a})u=0$, where $u(0,0)\neq 0$.
By Lemma~\ref{L:binomial_germ} the $\delta$-invariant is
\[
   \tfrac{1}{2}\bigl( (b-1)(d-a-1) + \gcd(b,d-a) - 1  \bigr)\,.
\]
A similar formula holds at the other points $[1:0:0]$ and $[0:1:0]$.
Summing these, equating with $p_a(C)=\binom{d-1}{2}$, and multiplying by $2$, we
obtain  
 \begin{multline}\label{Eq:delta-invariants}
 \qquad (d-1)(d-2)\ =\  d(a+b+c-3) -(ab+ac+bc) \\ + \gcd(a,d-c) 
    + \gcd(b,d-a) + \gcd(c,d-b) \,. \qquad 
 \end{multline}

We may assume that the coordinates and forms $s,t,\ell$ have been chosen so that 
$a$ is the maximum exponent and thus $d-a$ is the minimum.
We have $a\geq d-c$ and $b\geq d-a$, and there are two cases to consider
 \begin{equation}\label{Eq:alternative}
    c\geq d-b \qquad\mbox{or}\qquad d-b\geq c\,.
 \end{equation}
We study each case separately, beginning with $c\geq d-b$.

\begin{proposition}\label{P:solutions}
 The solutions to~$\eqref{Eq:delta-invariants}$ in the polytope \Blue{$P$}
 defined by the inequalities
 \begin{eqnarray}
   d-1 &\geq\ a\ \geq& d-c\nonumber\\
   d-1 &\geq\ b\ \geq& d-a\label{Eq:ineqs}\\
   d-1 &\geq\ c\ \geq& d-b\nonumber
 \end{eqnarray}
 are $(d{-}1,d{-}1,1)$, $(d{-}1,1,d{-}1)$, and $(1,d{-}1,d{-}1)$, for any $d\geq 2$, and 
 $(2,2,2)$ when $d=4$.
\end{proposition}

\begin{proof}
 Since $\gcd(\alpha,\beta)\geq\min(\alpha,\beta)$, for positive integers $\alpha$ and
 $\beta$, we may simplify~\eqref{Eq:delta-invariants} to obtain the inequality
%
\[
   (d-1)(d-2)\ \geq\ 
    (d-1)(a+b+c) -(ab+ac+bc) \,.
\]
 Let $\Blue{Q}=Q(a,b,c)$ be the symmetric quadratic form defined by the right hand side of
 this inequality. 
 We find its maximum values on the polytope $P$.
 First, the Hessian of $Q$ is
\[
   \mbox{hess}(Q)\ =\ 
    \left(\begin{matrix}0&-1&-1\\-1&0&-1\\-1&-1&0\end{matrix}\right)\ .
\]
 This has negative eigenvalue $-2$ with eigenvector $(1,1,1)$
 and positive eigenvalue $1$ with two-dimensional eigenspace $a+b+c=0$.
 In particular, $Q$ cannot take a maximum value in the interior of the polytope $P$ or in
 any of its facets.
 It can only take a maximum value in an edge that is parallel to the negative eigenspace
 $(1,1,1)$.

 The polytope $P$ is a symmetric bipyramid over the triangle whose vertices are
 \begin{equation}\label{Eq:triangle}
   (d{-}1,\,d{-}1,\,    1)\,,\ 
   (d{-}1,\,    1,\,d{-}1)\,,\ 
   (    1,\,d{-}1,\,d{-}1)\,.
 \end{equation}
 and with apices $(d{-}1,d{-}1,d{-}1)$ and $(\frac{d}{2},\frac{d}{2},\frac{d}{2})$.
 Since $P$ has no edge parallel to the negative eigenspace, $Q$ takes its maximum value 
 at vertices of $P$.

 The form $Q$ takes value $0$ at $(d{-}1,d{-}1,d{-}1)$, 
 $\frac{3}{4}d(d-2)$ at $(\frac{d}{2},\frac{d}{2},\frac{d}{2})$, and 
 $(d-1)(d-2)$ at the vertices~\eqref{Eq:triangle} of the triangle, and so the 
 vertices of the triangle give solutions.
 When $d=2$, $P$ degenerates to a point $(1,1,1)$, which is a solution
 to~\eqref{Eq:delta-invariants}.
 The only remaining possibility is that the point $(\frac{d}{2},\frac{d}{2},\frac{d}{2})$
 satisfies~\eqref{Eq:delta-invariants}.
 But then 
\[
    (d-1)(d-2)\ =\ \tfrac{3}{4}d(d-2)\,,
\]
in which case $d=4$ and so $(a,b,c)=(2,2,2)$ is the only other solution.
\end{proof}

If we take the alternative inequality in~\eqref{Eq:alternative}, 
$d-b\geq c$, then $Q$ becomes 
\[
   d(a+b+c-1) -a -ab-ac-bc\,.
\]
Replacing the third inequality in~\eqref{Eq:ineqs} by $d{-}b\geq c\geq 1$ defines a
tetrahedron with vertices 
\[
   (d{-}1,\, d{-}1,\, 1),\  (d{-}1,\,1,\,d{-}1),\ 
   (d{-}1,\, 1,\, 1),\ (\tfrac{d}{2},\tfrac{d}{2},\tfrac{d}{2})\,.
\]
Similar arguments as in the proof of Proposition~\ref{P:solutions} give the additional
solution $(d{-}1,1,1)$ to~\eqref{Eq:delta-invariants}.
By symmetry, we also obtain  solutions $(1,d{-}1,1)$ and $(1,1,d{-}1)$.\medskip

We write the curves $\gamma$ corresponding to these solutions.
The solution $(a,b,c)=(2,2,2)$ gives the expression  
\[
    [s^2t^2\,:\, t^2\ell^2\,:\, s^2\ell^2]\,,
\]
for $\gamma$ and the solutions $(a,b,c)=(d{-}1,d{-}1,1)$ give the expressions
 \begin{equation}\label{E:forms}
    [s^{d-1}t\,:\, t^{d-1}\ell\,:\, s^{d-1}\ell]\,.
 \end{equation}
The other two symmetric solutions give equivalent curves.
Lastly, the solutions $(a,b,c)=(d{-}1,1,1)$ give the expressions
\[
    [s^{d-1}t\,:\, t\ell^{d-1}\,:\, s^{d-1}\ell]\,,
\]
which become the expressions~\eqref{E:forms} under 
$x\leftrightarrow z$ and $t\leftrightarrow \ell$.

%
%
\subsection{Reduction to curves of type I}

%
%
\subsubsection{Quadratic Cremona transformations} 
We will show how curves of types II---VII are equivalent to curves of type I through
quadratic Cremona transformations.
We will sometimes use the non standard quadratic Cremona transformation
\[
   [x\colon y\colon z]\ \longmapsto  [z^2\colon xz\colon xy]\,.
\]
Permuting the variables gives five other non standard quadratic Cremona transformations.

%
%
\subsubsection{Curves of type II}
Suppose that $\gamma=[s^at^b\ell^{d-a-b}:s^{d-c}t^c:\ell^d]$ has type II.
We show that this may be transformed into a curve of type I by induction on $d$.
We will either transform $\gamma$ into a curve of type I or one of type II of lower
degree. 
Since $a+b<d=(d-c)+c$, interchanging $s$ and $t$ if necessary, we may assume that
$b<c$.
Applying the standard Cremona $[xy:xz:yz]$ transformation and removing the common factor of
$t^b\ell^{d-a-b}$ gives
\[
   [s^{a+d-c}t^{b+c}\ell^{d-a-b}\,:\, s^at^b\ell^{2d-a-b}
    \,:\, s^{d-c}t^c\ell^d]\ =\ 
   [s^{a+d-c}t^c\,:\, s^a\ell^d\,:\, s^{d-c}t^{c-b}\ell^{a+b}]\,.
\]
There remains a common power of $s$ that we can remove.
There are three cases to consider.

\begin{enumerate}
 \item If $a>d-c$, then we remove the common factor of $s^{d-c}$ to obtain a curve of
       type I.
 \item If $a=d-c$, then $d=a+c$ and we remove the common factor of $s^a=s^{d-c}$ to
       obtain 
 \[
      [s^at^c\,:\, \ell^{a+c}\,:\, t^{c-b}\ell^{a+b}] \ =\ [x\,:\,y\,:\,z]\,.
 \]
      We now apply the non standard Cremona transformation $[z^2:xz:xy]$ and remove the common
      factor $t^{c-b}\ell^{a+b}$ to obtain
 \begin{equation}\label{Eq:one_case_of_many}
      [t^{2c-2b}\ell^{2a+2b}\,:\, s^at^{2c-b}\ell^{a+b}\,:\, s^at^c\ell^{a+c}]
       \ =\ 
      [t^{c-b}\ell^{a+b}\,:\, s^at^c\,:\,s^at^b\ell^{c-b}]\,.
 \end{equation}
    If $b<c-b$, then we remove the common factor $t^b$ to obtain a curve of type I.
    If $b=c-b$, then we remove $t^b$ to get
 \[
     [\ell^{a+b}\,:\, s^at^b\,:\, s^a\ell^b]\ =\ [x\,:\,y\,:\,z]\,.
 \]
   Applying the non standard Cremona transformation $[z^2\,:\,xy\,:\,yz]$ to get
 \[
    [s^{2a}\ell^{2b}\,:\, s^at^b\ell^{a+b}\,:\, s^{2a}t^b\ell^b]\ =\ 
    [s^a\ell^b\,:\, t^b\ell^a\,:\, s^at^b]\,,
 \]
  which has type I.
  Finally, if $b>c-b$, then we remove the common factor of $t^{c-b}$
  from~\eqref{Eq:one_case_of_many} to obtain
 \[
     [\ell^{a+b}\,:\, s^at^b\,:\,s^at^{2b-c}\ell^{c-b}]\,,
 \]
  which has type II and degree $a+b<a+c=d$.

 \item If $a<d-c$, then we instead apply the non standard Cremona transformation
    $[x^2 :  yz : xz]$ to $\gamma$ and remove the common factor $s^at^b\ell^{d-a-b}$
       to obtain 
  \[ 
     [s^at^b\ell^{d-a-b}\,:\, s^{d-c-a}t^{c-b}\ell^{a+b}\,:\, \ell^d]\,.
  \]
    Removing the final common factor $\ell^{\min\{d-a-b,a+b\}}$ gives another curve of
    type II, but of lower degree.
\end{enumerate}

%
%
\subsubsection{Curves of type III}
Suppose that $\gamma=[s^at^{d-a}:t^b\ell^{d-b}:\ell^d]$ has type III.
We apply the standard Cremona transformation $[xy:xz:yz]$ and remove the common factor $\ell^{d-b}$ to
obtain 
\[
   [s^a t^{b+d-a}\,:\, s^at^{d-a}\ell^b\,:\,t^b\ell^d]\,.
\]
There remains a common power of $t$ that we can remove.
There are three cases to consider.

\begin{enumerate}
 \item If $b>d-a$, we factor out $t^{d-a}$ to get the type I curve, 
\[
   [s^at^b\,:\, s^a\ell^b\,:\, t^{a+b-d}\ell^d]\,.
\]

 \item If $b=d-a$, we instead apply the non standard Cremona transformation 
       $[y^2 : xy : xz]$ to $\gamma$ and factor out $t^b\ell^a$ to  get the type I
       curve, 
\[
      [t^b\ell^a\,:\, s^at^b\,:\, s^a\ell^b]\,.
\]

 \item If $b<d-a$, we factor out $t^b$ to obtain
\[
    [s^at^{d-a}\,:\, s^at^{d-a-b}\ell^b\,:\, \ell^d]\,,
\]
    which has type II, and we have already shown how to reduce a curve of type II to a
    curve of type I.
\end{enumerate}

%
%
\subsubsection{Curves of type IV}
Suppose that $\gamma=[s^at^{d-a}:s^{d-b}t^b:\ell^d]$ has type IV.
We may assume that $a>d-b$ and thus $b>d-a$.
We apply the standard Cremona transformation $[yz:xz:xy]$ and remove the common factor of
$s^{d-b}t^{d-a}$, to get the type I curve, 
\[
   [t^{a+b-d}\ell^d\,:\,s^{a+b-d}\ell^d\,:\, s^at^b]\,.
\]

%
%
\subsubsection{Curves of type V}
Suppose that $\gamma=[s^at^b\ell^{d-a-b}:t^d:\ell^d]$ has type V.
If we apply the standard Cremona transformation $[yz:xz:xy]$, we get the type I curve,
\[
   [t^d\ell^d\,:\,s^at^b\ell^{2d-a-b}\,:\, s^at^{d+b}\ell^{d-a-b}]\ =\ 
   [t^{d-b}\ell^{a+b}\,:\, s^a\ell^d\,:\, s^at^d]\,.
\]

%
%
\subsubsection{Curves of type VI}

Suppose that $\gamma=[s^at^{d-a}:t^d:\ell^d]$ has type VI.
If we apply the standard Cremona transformation $[yz:xz:xy]$, we get the type I curve,
\[
   [t^d\ell^d\,:\,s^at^{d-a}\ell^d\,:\, s^at^{2d-a}]\ =\ 
   [t^a\ell^d\,:\, s^a\ell^d\,:\, s^at^d]\,.
\]

%
%
\subsubsection{Curves of type VII}

Suppose that $\gamma=[s^d:t^d:\ell^d]$ has type VII. 
If we apply the standard Cremona transformation, we get the type I curve,
\[
   [t^d\ell^d\,:\,s^d\ell^d\,:\, s^dt^d]\,.
\]

%
\section{Singularities of polynomials}\label{S:technical}
%

Fix a form $F$ with prime factorization 
$F=F_{1}^{n_{1}}F_{2}^{n_{2}}\cdots F_k^{n_k}$.
Write $\sqrt{F}=F_{1}F_{2}\cdots F_{r}$ for its square free part, the product of its prime
factors.  In the following theorem we do not distinguish between the
curves $F=0$ and $\sqrt {F}=0$.\medskip 

\noindent{\bf Theorem~\ref{Th:singularities}.}
{\it 
  If a form $F$ coprime to $xyz$ defines a toric polar
  Cremona transformation, then the curve $F=0$ has at most one singular point
  outside the coordinate lines, and if there is such a point, then all factors of $F$ are
  contracted. Furthermore, if this singular point is ordinary,
 then $F$ has at most two distinct factors through the singular point.
}\medskip

We prove Theorem~\ref{Th:singularities} by studying the resolution of base points of the
toric polar linear system $T(F)$ at a singular point $p$ on $\sqrt {F}=0$ not lying on the
coordinate lines.  
In the resolution, there is a tree of exceptional rational curves lying over $p$.
We show that the leaves of this tree are exceptional curves above $p$
that are not contracted by the lift of the 
toric polar map, but are components of the lift of $\sqrt{F}=0$.
This implies that there is at most one such leaf and its exceptional
curve has multipicity 1, and that all other
components of this lift, including the strict transforms of 
the components of $\sqrt{F}=0$, are contracted by the toric polar map.
Thus there is at most one such singular point, and if it is ordinary,
then $F$ has two branches at this point.

Any common factor in $T(F)=\langle xF_{x},yF_{y}, zF_{z} \rangle$ is a multiple component of $F$.   
Indeed, if
$F=F_{1}^{n_{1}}F_{2}^{n_{2}}\cdots F_{r}^{n_{r}}$ is the prime factorization of $F$, 
then 
\[
   \Blue{G}\ :=\  \gcd(xF_x,yF_y,zF_z)\ =\ 
    F_{1}^{n_{1}-1}F_{2}^{n_{2}-1}\cdots F_{r}^{n_{r-1}}\,.
\]
is a common factor of $xF_{x},yF_{y}, zF_{z}$.
Removing this factor we get the vector space 
 \begin{equation}\label{Eq:RLS}
   \Blue{\sqrt{T(F)}}\ :=\ \langle F^x,F^y,F^z\rangle\ :=\ 
   \left\langle \frac{xF_x}{G},\frac{yF_y}{G},\frac{xF_x}{G}\right\rangle,
 \end{equation}
where $\Blue{F^{x}}:=n_{1}xF_{1,x}\dotsb F_{r}+\dotsb+n_{r}xF_{1}\dotsb F_{r,x}$,
and the same for $F^y$ and $F^z$.
Notice that 
\[
   F^x+F^y+F^z\ =\ (n_{1}\deg(F_1)+\dotsb+n_{r}\deg(F_r))F_1 \dotsb F_r
   \ =\ \deg(F)\sqrt{F}\,.
\]
In particular any common factor for form in $\sqrt{T(F)}$ is a factor of $\sqrt{F}$
and is thus one of the $F_{i}$. 
But no $F_{i}$ is a common factor, so forms in $\sqrt{T(F)}$ are coprime.

%
%

Let $p$ be a multiple point of the curve $\sqrt{F}=0$ outside the coordinate lines.  It is  
a common zero of the forms in $F^x,F^y,F^z$, as well as all partial
derivatives of $\sqrt{F}$, and is therefore a base point for
 $\sqrt{T(F)}$.
Resolving this base point and possibly infinitely near base points
gives a tree of exceptional rational curves lying over $p$.
We are only concerned with leaves of this tree so we assume that
$p=p_{0}, p_{1},\dotsc, p_{r}$ are
successive infinitely near base points that we blow up   
to resolve the base locus of $\sqrt{L_{F}}$ lying over $p$.  
In particular we assume that $p_{1}$ lies on the exceptional curve over $p$, the point
$p_{2}$ lies on the exceptional curve over $p_{1}$, and etc.  and that there are no base points infinitely near to $p_{r}$.  
Thus there is a unibranched curve that is smooth at $p_r$ and passes through all these 
infinitely near points. 
These are some, but not necessarily all of the infinitely near base points at
$p$.

We denote by $\pi_{1}\colon S_{1}\to S_{0}=\P^2$ 
the blow up of the point $p_{0}$, and by $E_{0}$ the exceptional curve of this map.
Inductively we denote by $\pi_{i}:S_{i}\to S_{i-1}$ the blowup of the point $p_{i-1}\in
S_{i-1}$ and by $E_{i-1}$ the exceptional curve of this map. 
Write $E_{i}$ also for the total transform in $S_k$ for $k>i$ of the exceptional curve
$E_i$ of $\pi_{i+1} \colon S_{i+1}\to S_{i}$.
The map $\pi:S_{r+1}\to \P^2$ is then the composition of the blowups $\pi_{i}$
for $i=1,\dotsc,r+1$.   

Let $\mu_0(\sqrt{T(F)})$ be the minimal multiplicity at $p$ of a curve in $\sqrt{T(F)}$.
Then the linear system $\sqrt{T(F)}_{(1)}$  on $S_1$ is generated by the strict transforms of curves
in $\sqrt{T(F)}$ having multiplicity $\mu_0(\sqrt{T(F)})$ at $p$.
Set $\mu_1(\sqrt{T(F)})$  to be the minimal multiplicity at $p_1$ of a curve
in $\sqrt{T(F)}_{(1)}$.
Then the linear system $\sqrt{T(F)}_{(2)}$  on $S_2$ is generated by the strict transforms of
curves in $\sqrt{T(F)}_{(1)}$ having multiplicity $\mu_1(\sqrt{T(F)})$ at $p_1$. 
Inductively, we obtain linear systems $\sqrt{T(F)}_{(i)}$ with multiplicities
$\mu_i(\sqrt{T(F)})$ at $p_i$.

For any curve $C$ in $\sqrt{T(F)}$, define the 
\Blue{{\sl virtual transform} $C_{(i)}$} in $\sqrt{T(F)}_{(i)}$ for $i=1,...,r$, to be unique member of 
this linear system that is mapped by $\pi_{1}\circ\dotsb\circ\pi_{i}$ to $C$ in $\P^2$.
Thus the virtual transform $C_{(i)}$ of $C$ on $S_{i}$ is the sum of the strict
transform of $C_{(i-1)}$ and $\bigl(\mu_{i-1}(C_{(i-1)})-\mu_{i-1}(\sqrt{T(F)})\bigr)E_{i-1}$, where
$\mu_{i-1}(C_{(i-1)})$ is the multiplicity of $C_{(i-1)}$ at $p_{i-1}$. 

We consider the virtual transform $\sqrt{F}_{(r+1)}$ in
$\sqrt{T(F)}_{(r+1)}$, and claim that it contains the leaf $E_{r}$ as a component.
For this we follow the line of argument in \cite {CA}, Section 8.5. We 
compare multiplicities and show that the inequality
\[
  \mu_{r}\bigl(\sqrt{F}_{(r)}\bigr)\ \geq\  \mu_{r}\bigl(\sqrt{T(F)}\bigr)
\]
is strict.
 We reduce this to a local calculation at $p$. 
 First, let $L_P$ be the polar linear system of $F$ defined by the partial derivatives
 $\langle F_x,F_y,F_z \rangle$. 
 Let $\sqrt{L_{P}}$ be the linear system obtained by removing the fixed component of
 $L_{P}$. 
 By linearity, $F_{\ell}=aF_{x}+bF_{y}+cF_{z}$ is the partial derivative of $F$ with
 respect to the linear form $\ell=ax+by+cz$. 
 In the Euler relation $dF=xF_{x}+yF_{y}+zF_{z}$ locally at $p$, the
 coordinates $x,y,z$ are units.  
 Therefore, locally at $p$, a general form in the toric polar linear system 
 $\langle xF_{x},yF_{y},zF_{z}\rangle$  is a linear combination of $F$ and its partial
 derivative $F_{\ell}$ with respect to some linear form $\ell$ that vanishes at $p$.
 In particular, such a general form has the same multiplicities as $F_{\ell}$ 
 (compare~\cite {CA} Section 7.2 and in particular Remark 7.2.4).
 So we may compute  $\mu_{i}(T(F))$ and $\mu_{i}(\sqrt{T(F)})$ locally at $p$, replacing
 $F=0$  and $\sqrt{F}=0$ with their germs $f$ and $\sqrt{f}$ at $p$,  and considering
 their partials derivatives (polars) with respect to linear forms that vanish at $p$. 
 In particular $\mu_{i}(T(F))= \mu_{i}(f_{\ell})$ for a general polar $f_{\ell}$ with
 respect to a linear form $\ell$ that vanishes at $p$.  
 We also change coordinates so that $x,y$ are local coordinates at $p$, and
 let $f_{x}, f_{y}$ be the germs of the polars with respect to $x$ and $y$.  

 We now analyze these germs.  
 For any two germs $g,\gamma$ of curves at $p$, we write \Blue{$[g,\gamma]_{p}$} for their local
 intersection multiplicity at $p$. 
 If $\gamma$ is unibranched, then this is simply the order of vanishing of the pullback of
 $g$ along a local parameterization of $\gamma$.
 Let $f=f_{1}^{n_{1}}f_{2}^{n_{2}}\cdots f_{k}^{n_{k}}$ be the irreducible factorization of
 the germ $f$ of the curve $F=0$ at $p$.  
 Set $g=f_{1}^{n_{1}-1}f_{2}^{n_{2}-1}\cdots f_{k}^{n_{k}-1}$ and let  
 $\sqrt{f}=f_{1}f_{2}\cdots f_{r}=\frac fg$, and $\overline{f_{x}}=\frac {f_{x}}g$ and
 $\overline{f_{y}}=\frac {f_{y}}g$. 

\begin{lemma}\label{L:germ_mult}
 Let $f$ be the germ of the curve $F=0$ at $p$, and let $\gamma$ be any smooth germ
 through $p$. 
 Then we have $[f,\gamma]_{p}> \min\{[f_{x},\gamma]_{p},[f_{y},\gamma]_{p}\}$.  
 Furthermore, 
\[
   [\sqrt{f}, \gamma]_{p}\ >\ 
   \min\{[\overline{f_{x}},\gamma]_{p},[\overline{f_{y}},\gamma]_{p}\}\,.
\]
\end{lemma}


\begin{proof}
Let $t\mapsto (x(t), y(t))$ be a local parameterization of $\gamma$.  
If $[f,\gamma]_{p}=n$, then $f(x(t),y(t))=t^nu$ for some invertible series $u$.  
Taking the derivative, we have 
\[
    nt^{n-1}u+t^n \frac{du}{dt}\ =\ 
    f_x(x(t),y(t))\frac{dx}{dt} + f_y(x(t),y(t))\frac{dy}{dt}
\]
so 
\[
    n-1\ \geq\  \min\{[f_{x},\gamma]_{p},[f_{y},\gamma]_{p}\}\,.
\]
Now, if $m_{i}=[f_{i}, \gamma]_{p}$, then $n=\sum_{i=1}^rm_{i}n_{i}$ and 
 \begin{eqnarray*}
   f_x(x(t),y(t))&=& n_{1}f_{1,x}f_{1}^{n_{1}-1}f_{2}^{n_{2}}\dotsb
   f_{k}^{n_{k}}+\dotsb+n_{k}f_{1}^{n_{1}}f_{2}^{n_{2}}\dotsb f_{k}^{n_{k}-1}f_{k,x}\\
  &=& 
   f_{1}^{n_{1}-1}f_{2}^{n_{2}-1}\dotsb f_{k}^{n_{k}-1}(n_{1}f_{1,x}f_{2}\dotsb
   f_{k}+\dotsb+n_{k}f_{1}f_{2}\dotsb f_{k,x})\,.
 \end{eqnarray*}
Similarly for $f_{y}$, so
\[
   n-1\ \ge\ \min\{[f_{x},\gamma]_{p},[f_{y},\gamma]_{p}\}\ =\ \sum_{i=1}^k m_{i}(n_{i}-1)
     +\min\{[\overline{f_{x}},\gamma]_{p},[\overline{f_{y}},\gamma]_{p}\}\,.
\]
 Therefore the strict inequality 
 $[\sqrt{f}, \gamma]_{p}> \min\{[\overline{f_{x}},\gamma]_{p},[\overline{f_{y}},\gamma]_{p}\}$ also holds
 at $p$.   
\end{proof}

\begin{lemma}\label{L:inequality}
Let $L= \langle g,h \rangle$ be a linear system of germs of curves on
a smooth surface $S$ and assume that $p=p_{0},...,p_{r}$ is a sequence
of infinitely near  base points for the linear system. 
 Let $f$ be a germ of a curve in $L$ whose virtual transform $f_{(i)}
$ in $L_{(i)}
$ has multiplicity $\mu_{i}(f_{(i)}
)$ at the point $p_{i}$, 
 and let $\mu_{i}(L)$ be the multiplicity of the linear system $L_{(i)}
$ at $p_{i}$.
 Assume that  for any smooth germ $\gamma$ on $S$ through $p$ the local intersection numbers satisfy:
  $$[f, \gamma]_{p}\ >\ 
   \min\{[g,\gamma]_{p},[h,\gamma]_{p}\}\,.$$
Then we have strict inequalities
 $\mu_{i}(f_{(i)})>\mu_{i}(L)$ for each $i=0,1,\dotsc,r$.
\end{lemma}

\begin{proof}
The effective multiplicities at the points $p=p_{0}, p_{1},\dotsc,p_{r}$ of the strict
transforms of the polar germs coincide for all but a finite number of members in the pencil
$\langle g, h\rangle$. 
Changing variables if necessary, we may assume that the multiplicity
sequence for the germs $g$ and $h$ coincide and are equal to that of the 
linear system: $\mu_{i}:=\mu_{i}(L)$ for $i=0,\dotsc,r$. 
These multiplicities are by definition the virtual multiplicities of $f$ with respect to
the linear system $L$.   
At $p=p_{0}$ the multiplicity $\mu_{0}$  differs from the multiplicity $e_{0}(f)$ of $f$.
If $\gamma$ is a smooth germ at $p$ that avoids all tangent directions of $f$, then  
$[f,\gamma]=e_{0}(f)$, and by assumption,
 \[
    \mu_{0}(F)\ =\ e_{0}(f)\ =\ [f,\gamma]\ >\ \min\{[g,\gamma],[h,\gamma]\}
     \ =\ \mu_{0}(L)\,.
\]  
Inductively, consider the virtual transform $f_{(i)}$ of $f$ on $S_{i}$ and choose a
unibranched germ $\gamma$ through the sequence of points $p_{0},\dotsc,p_{i}$ that is smooth
at $p_{i}$ and avoids all the tangent directions of $f_{(i)}$ at $p_{i}$.  
Let $e_{j}(\gamma)$ be the multiplicity of the strict transform $\gamma_{j}$ of
$\gamma$ at the $p_{j}$, for $j=0,\ldots,i$. 
Then, by assumption, 
\[
   [f,\gamma]\ >\ \min\{[g,\gamma],[h,\gamma]\}\ \geq\ \sum_{i=0}^{i-1}
   \mu_{j}e_{j}(\gamma)+\mu_{i}\,.
\]
 On the other hand, if $e_{j}(f)$ is the multiplicity of the strict transform $f_{j}$ on
 $S_{j}$ of $f$ at $p_{j}$  and $\mu_{i}(f_{(i)})$ is the multiplicity of the virtual
 transform $f_{(i)}$ of $f$ at $p_{i}$, then  
\[
   f_{(i)}\ =\ f_{i}-\sum_{j=0}^{i-1}(\mu_{j}-e_{j}(f))E_{j}
\]
while $e_{j}(\gamma)=[E_{j},\gamma_{j}]$, so
 \begin{eqnarray*}
   [f,\gamma] &=&\sum_{j=0}^{i-1}e_j(f) e_{j}(\gamma)+[f_{i},\gamma_{i}]\\
  &=& \sum_{j=0}^{i-1} \mu_{j}e_{j}(\gamma)+
   \sum_{j=0}^{i-1}(e_{j}(f)-\mu_{j})e_{j}(\gamma)+[f_{(i)},\gamma_{i}] 
   +\sum_{j=0}^{i-1} (\mu_{j}-e_{j}(f))[E_{j},\gamma_{j}]\\
   &=& \sum_{j=0}^{i-1}\mu_{j}e_{j}(\gamma)+\mu_{i}(f_{(i)})\,.
\end{eqnarray*}
In particular, $\mu_{i}(f_{(i)}) > \mu_{i} = \mu_{i}(L)$. 
\end{proof}

Lemma~\ref{L:inequality} applied to the curve $\sqrt{F}$ in the linear system
$\sqrt{T(F)}$ yields the following corollary.

\begin{cor} \label{C:leaves}
 Let $p$ be a multiple point of $\sqrt{F}$ outside the coordinate
 lines, and let $p_{r}$ be a base point of $\sqrt{T(F)}$ infinitely
 near to $p$, such that $\sqrt{T(F)}$ has no base points infinitely
 near to $p_{r}$. 
 Then the virtual transform of $\sqrt{F}=0$ in the linear system
 $\sqrt{T(F)}_{(r+1)}$ on the blowup of $S_{r}$ at $p_{r}$ contains
 the exceptional curve  $E_{r}$ as a component.
 Furthermore, the restriction of the linear system  $\sqrt{T(F)}_{(r+1)}$ to the exceptional curve $E_{r}$ has 
 degree equal to $\mu_{r}\bigl( \sqrt{T(F)} \bigr)$, the multiplicity of $\sqrt{T(F)}$ at $p_{r}$.
\end{cor}
\begin{proof} 
 Since the virtual multiplicity of $\sqrt{F}$ at $p_{r}$ is strictly greater than the multiplicity of the linear system by 
 Lemmas \ref{L:germ_mult} and \ref{L:inequality}, the first part follows.  
 The multiplicity of $\sqrt{T(F)}_{(r)}$ at $p_{r}$ is precisely the number of intersection points between
 the general member of $\sqrt{T(F)}_{(r+1)}$  and the exceptional curve $E_{r}$, so the second part also follows.
 \end{proof}

\begin{proof}[Proof of Theorem $\ref{Th:singularities}$]
The toric polar linear system $T(F)$ is equivalent to $\sqrt{T(F)}$
and $\sqrt{F}=0$ belongs to the later system. Assume that $p$ is a
singular point of $\sqrt{F}=0$ outside the coordinate lines. The point
$p$ is then a base point of $\sqrt{T(F)}$. The set of infinitely near
base points of $\sqrt{T(F)}$ at $p$ is finite, so it has at least one
point $p_{r}$ without further infinitely near base points. 
 By Corollary~\ref{C:leaves}, the exceptional curve $E_{r}$ on
 $S_{r+1}$ of this point is a component of the virtual transform
 $\sqrt{F}_{(r+1)}$ on $S_{r+1}$.  
 Since the linear system $\sqrt{T(F)}$ defines a birational map, the
 restriction of its base point free lift $\sqrt{T(F)}_{(r+1)}$ to
 $E_{r}$ must have degree $0$ or $1$.  But this degree is $\mu_{r}>0$,
 so $\mu_{r}=1$ and $E_{r}$ must be mapped isomorphically to a line. 
 Therefore all other components must be contracted and there can be no further multiple
 points of $\sqrt{F}$ outside the coordinate lines.   
 At an ordinary multiple point $p$ of $\sqrt{F}$ the multiplicity of
 $\sqrt{T(F)}$ is one less than the multiplicity of $\sqrt{F}$, so the
 last part follows. 
 \end{proof} 

\providecommand{\bysame}{\leavevmode\hbox to3em{\hrulefill}\thinspace}
\providecommand{\MR}{\relax\ifhmode\unskip\space\fi MR }
\providecommand{\MRhref}[2]{%
  \href{http://www.ams.org/mathscinet-getitem?mr=#1}{#2}
}
\providecommand{\href}[2]{#2}

\end{document}